\documentclass[a4paper,11pt]{article}
\pdfoutput=1

\usepackage[a4paper,tmargin=3truecm,bmargin=3truecm,rmargin=2.5truecm,
lmargin=2.5truecm,twoside,verbose=true]{geometry}
\usepackage{cancel,graphicx}

\usepackage{amsmath,amssymb}
\usepackage[amsmath, hyperref, thmmarks]{ntheorem}
\usepackage[all]{xypic}
\usepackage[pdftex]{hyperref}
\usepackage[english]{babel}

\numberwithin{equation}{section}

\theoremsymbol{}
\theorembodyfont{\slshape}
\theoremheaderfont{\normalfont\bfseries}
\theoremseparator{}
\newtheorem{Theorem}{Theorem}[section]
\newtheorem{theorem}[Theorem]{Theorem}

\newtheorem{proposition}[Theorem]{Proposition}

\newtheorem{lemma}[Theorem]{Lemma}

\theorembodyfont{\upshape}
\theoremsymbol{\ensuremath{\blacklozenge}}

\newtheorem{remark}[Theorem]{Remark}

\newtheorem{definition}[Theorem]{Definition}
\theoremstyle{nonumberplain}
\theoremheaderfont{\scshape}
\theorembodyfont{\normalfont}
\theoremsymbol{\ensuremath{\blacksquare}}

\newtheorem{proof}{Proof}
\qedsymbol{\ensuremath{_\blacksquare}}
\theoremclass{LaTeX}

\newcommand\caC{{\mathcal C}}
\newcommand\caD{{\mathcal D}}
\newcommand\caE{{\mathcal E}}
\newcommand\caF{{\mathcal F}}

\newcommand\caL{{\mathcal L}}
\newcommand\caP{{\mathcal P}}

\newcommand\caT{{\mathcal T}}

\newcommand\caS{{\mathcal S}}
\newcommand\caO{{\mathcal O}}

\newcommand\bfm{{\mathsf{\mathbf m}}}
\newcommand\gone{{ \mathchoice {1\mskip-4mu\mathrm{l} } {1\mskip-4mu\mathrm{l} }{1\mskip-4.5mu\mathrm{l} } {1\mskip-5mu\mathrm{l}} }}
\newcommand\gR{{\mathbb R}}

\newcommand\gC{{\mathbb C}}

\newcommand\gS{{\mathbb S}}

\newcommand\xz{{z}}
\newcommand\yz{{\tilde z}}
\newcommand\xw{{w}}
\newcommand\yw{{\tilde w}}
\newcommand\that{\hat}
\newcommand\thatp{\hat}

\newcommand\algA{{\mathbf A}}

\newcommand\algB{{\mathbf B}}

\newcommand\ehH{\mathcal H}

\newcommand\bb{{\text{\textup{b}}}}
\newcommand\ka{{\mathfrak a}}
\newcommand\kb{{\mathfrak b}}

\newcommand\kq{{\mathfrak q}}
\newcommand\kg{{\mathfrak g}}

\newcommand\kn{{\mathfrak n}}
\newcommand\ks{{\mathfrak s}}

\newcommand\kM{{\mathfrak M}}

\newcommand\eps{{\varepsilon}}
\newcommand\ad{{\text{\textup{ad}}}}
\newcommand\Ad{{\text{\textup{Ad}}}}

\newcommand\fois{\mathord{\cdot}}
\DeclareMathOperator{\tr}{Tr} 

\newcommand\Der{{\text{\textup{Der}}}}

\newcommand\dd{{\text{\textup{d}}}}

\newcommand\norm{\mathord{\parallel}}

\newcommand\ee{{\epsilon}}

\title{Non-formal star-exponential on contracted one-sheeted hyperboloids\footnote{Work
supported by the Belgian Interuniversity Attraction Pole (IAP) within the framework ``Dynamics, Geometry and Statistical Physics'' (DYGEST).}}
\author{Pierre Bieliavsky$^a$, Axel de Goursac$^{a}$, Yoshiaki Maeda$^b$, Florian Spinnler$^{a,c}$}
\begin{document}

\maketitle

\vspace*{-1cm}
\begin{center}
\textit{$^a$Institut de Recherche en Math\'ematique et Physique, Universit\'e Catholique de Louvain\\ Chemin du Cyclotron, 2, 1348 Louvain-la-Neuve, Belgium\\
e-mail: \texttt{Pierre.Bieliavsky@uclouvain.be, Axelmg@melix.net, Florian.Spinnler@uclouvain.be}}\\
\vskip .5cm
\textit{$^b$ Tohoku Forum for Creativity, Tohoku University\\
 2-1-1, Katahira, Aoba-Ku, Sendai, 980-8577 Japan\\
    e-mail: \texttt{Yoshimaeda@m.tohoku.ac.jp}}\\
    \vskip .5cm
\textit{$^c$ D\'epartement de Math\'ematiques, Facult\'e des Sciences\\ Universit\'e Libre de Bruxelles, Boulevard du Triomphe\\ 1050 Bruxelles, Belgium}\\
\end{center}%

\vskip 2cm

\begin{abstract}
In this paper, we exhibit the non-formal star-exponential of the Lie group $SL(2,\gR)$ realized geometrically on the curvature contraction of its one-sheeted hyperboloid orbits endowed with its natural non-formal star-product. It is done by a direct resolution of the defining equation of the star-exponential and produces an expression with Bessel functions. This yields a continuous group homomorphism from $SL(2,\gR)$ into the von Neumann algebra of multipliers of the Hilbert algebra associated to this natural star-product. As an application, we prove a new identity on Bessel functions.
\end{abstract}
\vskip 4cm


\pagebreak

\tableofcontents
\vskip 1cm

\section{Introduction}
\label{sec-intro}
Deformation quantization, initiated in \cite{Bayen:1978}, consists in deforming the pointwise product of the commutative algebra of smooth functions $\caC^\infty(M)$ on a Poisson manifold $M$ into a noncommutative star-product $\star_\theta$ depending on a deformation parameter $\theta$. Formal deformation quantizations were intensively studied \cite{Omori:1991,Lecomte:1992,Fedosov:1994,Gutt:2000} and definitely classified in \cite{Kontsevich:2003}. In the non-formal setting, there exist some examples of deformation of groups and their actions like Abelian Lie groups $\gR^{2n}$ \cite{Rieffel:1993}, Abelian Lie supergroups \cite{Bieliavsky:2010su,deGoursac:2014kv}, K\"ahlerian Lie groups \cite{Bieliavsky:2010kg}, Abelian $p$-adic groups \cite{Gayral:2014ad}, deformations of $\gC^{2n}$ in holomorphic \cite{Omori:2000,Bieliavsky:2002ma} or resurgent \cite{Garay:2013gya} context, deformations of $SU(1,n)$ \cite{Bieliavsky:2008mv,Korvers:2014}, but no general classifying theory is available.

Associated to star-products and following \cite{Fronsdal:1978}, the notion of star-exponential \cite{Bayen:1978,Bayen:1982,Arnal:1988} plays an important role for the study of deformation quantization, for giving access to spectrum of operators \cite{Cahen:1984,Cahen:1985}, for the link with representation theory. In the non-formal context, star-exponential of quadratic functions were explicitly computed \cite{Omori:2000,Omori:2011tr} for the Moyal-Weyl product. Applications to harmonic analysis such as character formula or Fourier transformation can be obtained by computing non-formal star-exponential of momentum maps of some Lie group's action. This was performed for nilpotent Lie groups in \cite{Arnal:1990} by using the Moyal-Weyl product. By using Berezin and Weyl quantizations, this program of star-representations was achieved in the case of unitary irreducible representations of compact semisimple Lie groups \cite{Arnal:1988cg}, of holomorphic discrete series \cite{Arnal:1989cg} (see also \cite{Bieliavsky:1999st} for $SL_2(\gR)$) and principal series of semisimple Lie groups \cite{Cahen:1996st} (see also \cite{Arnal:1996ya}). 


However, one can wonder wether it is possible to construct non-formal star-exponentials for star-products that are geometrically more natural for the orbits of the group. In this spirit, the non-formal star-exponential of K\"ahlerian Lie groups with negative curvature was exhibited in \cite{Bieliavsky:2013sk} for invariant star-products on their coadjoint orbits, with application to the construction of an adapted Fourier transformation.

In this paper, we are interested in the one-sheeted hyperboloid orbits of $SL_2(\gR)$ \cite{Lang:1975}, also called two-dimensional anti-de Sitter space, $AdS_2\,:=\,SL_2(\gR)/SO(1,1)$. To compute its star-exponential, we want to dispose of a non-formal $\mathfrak{sl}_2(\gR)$-covariant star-product geometrically adapted to $AdS_2$, and to this aim, we will look at its natural contraction. Let us first show that this contraction of $AdS_2$ corresponds locally to the symmetric space $M\,:=\,SO(1,1)\ltimes\gR^2/\gR$ called Poincar\'e coset. The global picture is however given in the conclusion of this paper but it is not needed now.

\medskip

This curvature contraction is induced by the contraction of Lie algebras:
\begin{equation*}
\mathfrak{sl}_2(\gR)\longrightarrow\mathfrak{so}(1,1)\ltimes\gR^2
\end{equation*}
that corresponds to the limit $t\to0$ in the following three-dimensional real Lie algebra $\kg_t$ table:
\begin{equation*}
\left[H,E\right]=2E,\qquad \left[H,F\right]=-2F,\qquad \left[E,F\right]=t\,H\;,
\end{equation*}
where $\kg_t$ is isomorphic to $\kg_{1}\simeq\mathfrak{sl}_2(\gR)$ for every $t>0$,  while $\kg_0\simeq\mathfrak{so}(1,1)\ltimes\gR^2$. 

\noindent The above contraction of Lie algebras induces a geometric contraction at the level of naturally associated symmetric spaces. To see this, observe first that
setting
\begin{equation*}
\sigma_t(H)\;:=\;-H\quad\sigma_t(E+F)\;:=\;E+F\quad\sigma_t(E-F)\;:=\;F-E
\end{equation*}
defines an involutive automorphism of $\kg_t$ for every $t\in\gR$. When $t>0$, the adjoint orbit of the element $E+F$ in $\kg_t$ then realizes a symmetric space $M_t$ admitting
$(\kg_t,\sigma_t)$ as associated infinitesimal involutive Lie algebra. 

\noindent One gets a local chart on $M_t$ by considering the open orbit of the base point $E+F$ under the action the Iwasawa factor $\gS\;:=\;\exp(\gR H\oplus\gR E)$: 
\begin{equation*}
\varphi_{(t)}\;:\;\gR H\oplus\gR E\longrightarrow M_t:(a,\ell)\mapsto\Ad_{e^{aH}e^{\ell E}}(E+F)\;.
\end{equation*}
Within this local chart, the geodesic symmetry at the base point $E+F$ in $M_t$ corresponds to the map\footnote{The expression of the symmetry centered around another point is easily computed from the $\gS$-equivariance of the symmetric space structure.}
\begin{equation*}
s^{(t)}_{(0,0)}(a,\ell)\;:=\;(\,-a\,-\,\log(1-t\ell^2)\,,\,-\ell)
\end{equation*}
whose maximal domain is the open strip $|\ell|\,<\,t^{-1/2}$. 

\noindent Note that in the limit $t\to0$, the symmetry becomes global ($s^{(0)}_{(0,0)}=-\mbox{\rm id}$), realizing the canonical (non-flat) symmetric space structure on $M=M_0=SO(1,1)\ltimes\gR^2/\gR$:
\begin{equation*}
s_{(a,\ell)}^{(0)}(a',\ell')\;=\;(2a-a'\,,\,2\cosh(2(a-a'))\ell-\ell')\;.
\end{equation*}
All the $\star$-products on $M$ that are invariant under the symmetries $\{s^{(0)}_x\}_{x\in M}$ are known in a totally explicit way \cite{Bieliavsky:2008or}. Of course, while $\gS$-invariant, none of them is
$SL_2(\gR)$-invariant, even not at the infinitesimal level. However, the contraction procedure is somehow partially remembered at the limit $t\to0$:
every such $\star$-product turns out to be $\mathfrak{sl}_2(\gR)$-{\sl covariant} in the sense that the classical moment mapping 
\begin{equation}
\label{MOMENT}
\lambda^{(t)}\;:\;\mathfrak{sl}_2(\gR)\to C^\infty(M_t):X\mapsto\lambda^{(t)}_X
\end{equation}
when restricted to the image of the local chart still yields a Lie algebra homomorphism into the $\star$-product algebra at the limit $t\to0$. It means that if one endows $C^\infty(M)[[\theta]]$ with a $SO(1,1)\ltimes\gR^2$-invariant $\star$-product $\star_\theta^1$, the map
\begin{equation*}
\mathfrak{sl}_2(\gR)\to \left(C^\infty(M)[[\theta]],\star^1_\theta\right)\,:\,X\,\mapsto\,(\varphi_{(0)}^{-1})^\star\left(\varphi_{(t)}^\star\lambda^{(t)}_X\right)
\end{equation*}
is a homomorphism of Lie algebras for every $t>0$.

\medskip

In view of the fact that this contraction process is entirely canonical, it is tempting to study its possible relation with the representation theory of $SL_2(\gR)$ and in particular 
to investigate whether it reproduces the unitary irreducible representation series canonically associated to the $AdS_2$-orbits i.e. the principal series.

This is what is done in the present article within a non-formal star-product (i.e. operator algebraic) approach. More precisely, we here consider the non-formal $SO(1,1)\ltimes\gR^2$-invariant product $\star^1_\theta$ ($\theta\in\gR_0$) on $L^2(M)$ defined in \cite{Bieliavsky:2010kg}. Setting the Hilbert algebra $\algA_\theta\;:=\;(L^2(M), \star^1_\theta)$ and denoting by $\kM_\bb(\algA_\theta)$ its von Neumann algebra of bounded multipliers (see in particular \cite{deGoursac:2014mu}), we prove that the formal Lie map (\ref{MOMENT}), say at $t=1$, actually exponentiates 
to a weakly continuous group morphism:
\begin{equation*}
\caE_{\star_\theta^1}:SL_2(\gR)\to\kM_\bb(\algA_\theta)\;.
\end{equation*}

After a presentation of the geometric context in section \ref{sec-present}, the explicit expression of $\caE_{\star_\theta^1}$ on the generator $F$ in some coordinate chart $\Phi_\kappa$ is directly computed in terms of Bessel functions in section \ref{sec-direct} by constructing the spectral measure of the differential operator involved in the defining equation of the star-exponential. On another coordinate chart $\Psi_\kappa$ and with another star-product $\sharp_\theta$ (already considered in \cite{Cahen:1996st}), the star-exponential $\caE_{\sharp_\theta}$ is expressed in terms of the principal series representation $\caP^\theta$ associated with the $AdS_2$-orbit in section \ref{sec-principal}. 

Then, we want to relate both star-exponentials, $\caE_{\star_\theta^1}$ expressed with Bessel functions and $\caE_{\sharp_\theta}$ expressed with the principal series representation, so we need an intertwiner between the corresponding star-products. Three different methods are presented to obtain explicitly a unitary intertwining operator $W$ between $\sharp_\theta$ and $\star_\theta^1$ in section \ref{sec-inter}. In order to get unitarity of $W$, we need two copies of the range, namely $\algA_\theta\oplus\algA_{-\theta}$, so that we rather proved that the associated left-regular representation
\begin{equation*}
g\mapsto\big(\caE_{\star_\theta^1}(g)\star_\theta^1\,\bullet,\ \caE_{\star_{-\theta}^1}(g)\star_{-\theta}^1\,\bullet\big)
\end{equation*}
is unitarily equivalent with $\caP^\theta(g)\otimes\mbox{\rm id}$. In section \ref{sec-concl}, we get a nice geometric interpretation of these two copies in terms of the global curvature contraction of $AdS_2$. As an application, we end this article by deriving from the comparison of $\caE_{\star_\theta^1}(e^{tF})$ and $W(\caE_{\sharp_\theta}(e^{tF}))$  a new identity on Bessel functions.

\section{Star-products for Anti-deSitter space}
\label{sec-present}

\subsection{Adjoint orbits of $SL(2,\gR)$}
\label{subsec-orbit}

Let us fix the notations. We consider $\kg:=\mathfrak{sl}(2,\mathbb{R})$ the Lie algebra  of the group $G:=SL(2,\mathbb{R})$. We choose a basis of $\kg$ satisfying the following commutation relations:  
\begin{equation*}
[H,E]=2E,\qquad [H,F]=-2F,\qquad [E,F]=H.
\end{equation*}
The maximal compact subgroup of $G$ is isomorphic to $SO(2)$ and its Lie algebra is generated by the element $E-F$.

The Lie algebra $\kg$ can be (Cartan) decomposed as a direct vector space sum $\mathfrak{k}\oplus \mathfrak{P}$, where $\mathfrak{P}=\langle H,E+F \rangle $ is the orthogonal complement for the Killing form of the Lie algebra $\mathfrak{k}\equiv \langle E-F \rangle$ of $SO(2)$. The Killing form $\beta$ on $\kg$ is given by : $\beta(X,Y)=\frac{1}{8}\tr(\ad_X\ad_Y)$ (this normalization differs from the usual one). Furthermore, the Lie algebra $\kg$ can also be (Iwasawa) decomposed as 
\begin{equation*}
\kg=\mathfrak{k}\oplus \mathfrak{a} \oplus \mathfrak{n},
\end{equation*}
where $\mathfrak{a}$ is a maximal abelian subalgebra of $\mathfrak{P}$, and $\mathfrak{n}$ is the nilpotent algebra obtained as the sum of the positive root spaces of $\mathfrak{a}$, for some choice of positiveness on $\mathfrak{a}^{*}$. Here we choose $\mathfrak{a}=\langle H\rangle$ so that the roots are $-2H^{*}$, $0$ and $2H^{*}$, and with the natural order we get $\mathfrak{n}=\langle E\rangle$. The Iwasawa subgroup of $G=SL(2,\mathbb{R})$ will be denoted by $\gS$, its Lie algebra is exactly $\ks=\ka\oplus\kn$.

Let us consider the adjoint orbits of $SL(2,\gR)$. For $\kappa \in \mathbb{R}$, fix $o_{\kappa} := E+\kappa F \in \kg$ and $\mathcal{O}_{\kappa}=\Ad_{G}(o_{\kappa})$ the adjoint orbit of $o_{\kappa}$ under the adjoint action of $G$. The case $\kappa>0$ corresponds to the Anti-DeSitter space, $\kappa=0$ to the positive half-line along $E$, while $\kappa<0$ corresponds to the hyperbolic plane. We will assume in this paper that $\kappa>0$. The Kirillov-Kostant-Souriau (KKS) symplectic form defined on the adjoint orbits of $G$ can be written as follows
\begin{displaymath}
\omega_z^{\mathcal{O}_{\kappa}}(X^{*}_z,Y^{*}_z)\equiv\beta(z,[X,Y]),
\end{displaymath}		
where $\beta$ is the Killing form of $\kg$ as above and $X^{*}$ is the fundamental vector field associated to $X\in \kg$, defined at a point $z\in \kg$ on $f\in C^\infty(\kg)$ by $X^{*}_z f=\frac{\dd}{\dd t}_{|t=0}f(e^{-tX}z)$.

We will consider the coordinate system $\Phi_{\kappa}: \ks \rightarrow \mathcal{O}_{\kappa}$ on the orbit $\mathcal{O}_{\kappa}$, given in terms of the Iwasawa subgroup $\gS$:
\begin{equation}
\label{eq-chartphi}
\Phi_\kappa(aH+\ell E):= \Ad_{e^{aH}e^{\ell E}}(o_{\kappa})=\kappa \ell H+(1-\kappa\ell^2 )e^{2a}E+ \kappa e^{-2a}F.
\end{equation}
It turns out that $\Phi_{\kappa}$ is a Darboux chart, but its image only corresponds to half of the adjoint orbit (see Figure \ref{fig-AdS-habit}). As we normalize the Killing form such that $\beta(H,H)=1$, $\beta(E,F)=\beta(F,E)=\frac{1}{2}$, we have $\omega^{\Phi}:=\Phi_{\kappa}^*\omega_z^{\mathcal{O}_{\kappa}}=\kappa \dd a\wedge\dd \ell$. In this chart, the action of $\gS$ corresponds to the left-multiplication.

The adjoint action of $G$ on $\ks\simeq\caO_\kappa$ (in this coordinate system) is strongly hamiltonian with moment map $\lambda$ given explicitly by
\begin{equation*}
\lambda_H(a,\ell )=\ell\kappa,\qquad \lambda_E(a,\ell )=\frac{\kappa}{2} e^{-2a},\qquad \lambda_F(a,\ell)=\frac{(1-\kappa\ell^2)}{2}e^{2a}.
\end{equation*}

\begin{figure}[!htb]
  \centering
  \includegraphics[height=10cm]{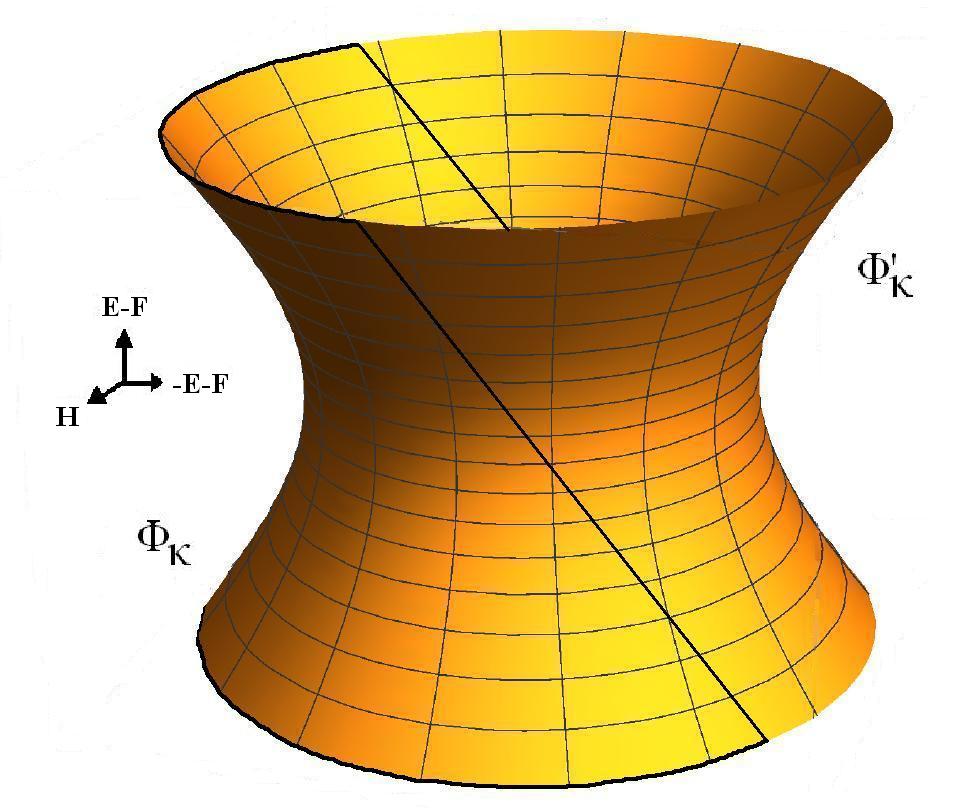}
  \caption[AdS]{\footnotesize{The two charts $\Phi_\kappa$ and $\Phi'_\kappa$ on $AdS_2$.}}
  \label{fig-AdS-habit}
\end{figure}


\subsection{Star-products and star-exponentials}
\label{subsec-starprod}

Set $\theta\in\gR^*$. The $\Phi_\kappa$ coordinate system yields a subspace $\ks$ of $\caO_\kappa$ isomorphic to $\gR^2$, on which we can define the Moyal product in the direction of the symplectic form $\kappa\dd a\wedge\dd \ell$. It is given by: for Schwartz functions $f_1,f_2\in \caS(\gR^2)$,
\begin{equation}
\big(f_1\star^0_\theta f_2\big)(a,\ell):=\frac{\kappa^2}{\pi^2\theta^2}\int f_1(a_1,\ell_1) f_2(a_2,\ell_2) e^{\frac{-2i\kappa}{\theta} (a_1\ell_2-a_2\ell_1+a_2\ell-a\ell_2+a\ell_1-a_1\ell)} \dd a_i\dd \ell_i\label{eq-star0}
\end{equation}
Such a Moyal star-product is $\kg$-covariant but not $\gS$-invariant. One prefers to deal with a star-product on $\gS\simeq\ks$ with symplectic form $\kappa\dd a\wedge\dd \ell$, which is $\gS$-invariant for the left action on $\gS$, or equivalently for the coadjoint action on $\ks$. It is also $\kg$-covariant, it was explicitly found in \cite{Bieliavsky:2002} and has the expression
\begin{multline}
(f_1\star_\theta^1 f_2)(a,\ell):=\frac{\kappa^2}{\pi^2\theta^2}\int \sqrt{\cosh(2(a_1-a_2))\cosh(2(a_2-a))\cosh(2(a-a_1))}\\
f_1(a_1,\ell_1)f_2(a_2,\ell_2)\, e^{\frac{i\kappa}{\theta}\Big(\sinh(2(a_1-a_2))\ell+\sinh(2(a_2-a))\ell_1+\sinh(2(a-a_1))\ell_2\Big)} \dd a_1\dd a_2\dd \ell_1\dd \ell_2.\label{eq-star1}
\end{multline}
Actually, it can be obtained by intertwining the Moyal product $\star_\theta^0$: $f_1\star^1_{\theta}f_2= T_{01}((T_{01}^{-1}f_1)\star_\theta^0(T_{01}^{-1}f_2))$ where the intertwiners (which are not $\gS$-equivariant) lie:
\begin{align}
&T_{01} f(a,\ell):=\frac{1}{2\pi}\int\ \sqrt{\cosh(\frac{\theta t}{\kappa})} e^{\frac{i\kappa}{\theta}\sinh(\frac{\theta t}{\kappa})\ell-i\xi t}f(a,\xi)\dd t\dd\xi\nonumber\\
&T^{-1}_{01} f(a,\ell)=\frac{1}{2\pi}\int\ \sqrt{\cosh(\frac{\theta t}{\kappa})} e^{-\frac{i\kappa}{\theta}\sinh(\frac{\theta t}{\kappa})\xi+it\ell}f(a,\xi)\dd t\dd\xi\label{eq-inter01}
\end{align}

\begin{remark}
\label{rmk-prodpoincare}
There is a $\gS$-equivariant diffeomorphism with the Poincar\'e coset $M:=SO(1,1)\ltimes\gR^2/\gR\simeq\gS$. In this point of view, the star-product $\star_\theta^1$ introduced above coincides with the natural star-product on $M$, namely the unique star-product invariant under the action of $SO(1,1)\ltimes\gR^2$ and of involution $(a,\ell)\mapsto (-a,-\ell)$ \cite{Bieliavsky:2008or}. It explains why we will consider it in the following and why we are interested in its associated star-exponential.
\end{remark}

It turns out that these products give rise to complete Hilbert algebras $(L^2(\gR^2),\star_\theta^0)$ and $(L^2(\gS),\star_\theta^1)$ for the usual complex conjugation and the standard scalar product. We call them Hilbert deformation quantizations \cite{deGoursac:2014mu} and we can consider the left von Neumann algebra (type $I_\infty$ factor), which will be expressed here as bounded multipliers of the Hilbert algebra.

Recall that a complex algebra with involution and scalar product is a Hilbert algebra if $\forall f_i\in\algA$,
\begin{equation*}
\langle f_2^*,f_1^*\rangle=\langle f_1,f_2\rangle,\qquad \langle f_1\fois f_2,f_3\rangle=\langle f_2,f_1^*\fois f_3\rangle,
\end{equation*}
if the map $\lambda_{f_1}:f_2\in\algA\mapsto f_1\fois f_2$ is bounded for the norm $\norm f\norm:=\sqrt{\langle f,f\rangle}$, and if the set $\{f_1\fois f_2,\ f_i\in\algA\}$ is dense in $\algA$.

Suppose that $\algA$ is a complete Hilbert algebra, it is in particular a Hilbert space. Its left von Neumann algebra can be expressed as the left part of bounded multipliers $\kM_\bb(\algA)$ \cite{deGoursac:2014mu} that are pairs $T=(L,R)$ of bounded operators on $\algA$ satisfying $\forall f_i\in\algA$,
\begin{equation}
f_1\fois L(f_2)=R(f_1)\fois f_2.\label{eq-multiplier}
\end{equation}
We have the equivalent characterization that $L$ is a bounded operator on $\algA$ satisfying 
\begin{equation}
L(f_1\fois f_2)=L(f_1)\fois f_2\quad\text{ and }\quad R(f)=(L^*(f^*))^*.\label{eq-multiplier2}
\end{equation}
Note that any unitary *-isomorphism $\Phi:\algA\to\algB$ between two complete Hilbert algebras can be extended to the bounded multipliers
\begin{equation}
\tilde\Phi:\kM_\bb(\algA)\to\kM_\bb(\algB),\label{eq-morphmult}
\end{equation}
by $\tilde\Phi(T):=(\Phi\circ L\circ\Phi^{-1},\Phi\circ R\circ \Phi^{-1})$.

\medskip

We note $\kM_{\star_\theta^0}(\gR^2)$ and $\kM_{\star_\theta^1}(\gS)$ the bounded multipliers algebras associated to the Hilbert algebras $(L^2(\gR^2),\star_\theta^0)$ and $(L^2(\gS),\star_\theta^1)$. These von Neumann algebras will be very useful as functional spaces to characterize the non-formal star-exponential. The intertwiner $T_{01}$ is actually a unitary *-isomorphism between these two Hilbert algebras. Even if the group $G=SL(2,\gR)$ does not stabilize the spaces $\gR^2$ and $\gS$ seen as part of the adjoint orbit $\caO_\kappa$ via the coordinate system $\Phi_\kappa$, we can consider the space generated by the moment maps $\lambda_H$, $\lambda_E$ and $\lambda_F$. Since the star-products are $\kg$-covariant, such a space has a Lie algebra structure for the star-commutator and it can also be denoted by $\kg$:
\begin{equation*}
[\lambda_H,\lambda_E]_{\star_\theta^j}=-2i\theta\lambda_E,\qquad [\lambda_H,\lambda_F]_{\star_\theta^j}=2i\theta\lambda_F,\qquad [\lambda_E,\lambda_F]_{\star_\theta^j}=-i\theta\lambda_H,
\end{equation*}
for $j=0,1$, and where one understands these functions $\lambda_X$ ($X\in\kg$) as unbounded operators acting by left $\star_\theta^j$-multiplication. It means that $\kg$ is a symmetry in the sense of \cite{deGoursac:2014mu} for the Hilbert deformation quantizations $(L^2(\gR^2),\star_\theta^0)$ and $(L^2(\gS),\star_\theta^1)$. Therefore, the (non-formal) star-exponential of this symmetry is well-defined \cite{Arnal:1990,Omori:2000}
\begin{equation*}
E_{\star_\theta^j}(\frac{i}{\theta}\lambda_X):=\sum_{k=0}^\infty \frac{1}{k!}\left(\frac{i}{\theta}\lambda_X\right)^{(\star_\theta^j k)},
\end{equation*}
as being a unitary (bounded) multiplier and we have $E_{\star_\theta^1}(\frac{i}{\theta}\lambda_X)=T_{01}E_{\star_\theta^0}(\frac{i}{\theta}\lambda_X)$ since $T_{01}\lambda_X=\lambda_X$. Moreover, this star-exponential satisfies the BCH property
\begin{equation*}
E_{\star_\theta^j}(\frac{i}{\theta}\lambda_X)\star_\theta^j E_{\star_\theta^j}(\frac{i}{\theta}\lambda_Y)=E_{\star_\theta^j}(\frac{i}{\theta}\lambda_{\text{BCH}(X,Y)})
\end{equation*}
if $\text{BCH}(X,Y)$ is defined. In particular, note that
\begin{align*}
&\text{BCH}(aH,\ell E)=aH+\frac{ae^a\ell}{\sinh(a)}E,\qquad \text{BCH}(aH,mF)=aH+\frac{ae^{-a}m}{\sinh(a)}F,\\
& \text{BCH}(\ell E,mF)=\frac{\text{Arccosh}(1+\frac12\ell m)}{\sqrt{\ell m+\frac14\ell^2m^2}}(\frac12\ell mH+\ell E+mF)\quad\text{if }\ell m>0.
\end{align*}
The $ax+b$ part of the symmetry ($\gS\subset G$) was studied in \cite{Bieliavsky:2013sk} and its star-exponential was explicitly computed
\begin{equation}
E_{\star^0_\theta}(\frac{i}{\theta}t\lambda_H)(a,\ell)=e^{\frac{i\kappa}{\theta}t\ell},\qquad E_{\star^0_\theta}(\frac{i}{\theta}t\lambda_E)(a,\ell)=e^{\frac{i\kappa}{2\theta} te^{-2a}}.\label{eq-starexphe}
\end{equation}
In the following, we want to compute explicitly the star-exponential of the last generator $\lambda_F$ of the symmetry $\kg$, to express the link between this star-exponential and the principal series representation of $SL(2,\gR)$ and to integrate the star-exponential at the level of the group.

\section{Direct computation of the star-exponential}
\label{sec-direct}

\subsection{Resolution of the equation}

Let us solve the defining equation of the star-exponential of $\lambda_F$
\begin{equation*}
\partial_t E_{\star^0_{\theta}}(\frac{i}{\theta}t\lambda_F)(a,\ell)=\frac{i}{\theta}\lambda_F\star_\theta^0 E_{\star^0_{\theta}}(\frac{i}{\theta}t\lambda_F)(a,\ell)
\end{equation*}
with initial condition $E_{\star^0_{\theta}}(\frac{i}{\theta}t\lambda_F)\equiv 1$ for $t=0$. We use the star-product $\star_\theta^0$ instead of the natural $\star_\theta^1$ for more simplicity in this equation, but we know explicitly the intertwiner $T_{01}$ \eqref{eq-inter01} between both star-products. Actually, we will solve this equation with a general initial condition $f\in L^2(\gR^2)$ to be able to turn it into a multiplier. Note that due to the expression of the Moyal product \eqref{eq-star0}, this equation is not a PDE but it contains integrals. To obtain a PDE, one can first perform a partial Fourier transformation with respect to the variable $\ell$
\begin{equation}
\caF_\ell(f)(a,\xi):=\int_{\gR} f(a,\ell)e^{-i\xi \ell}\dd \ell \label{eq-partfourier}
\end{equation}
Indeed, we compute that
\begin{equation*}
\caF_\ell(\lambda_F\star^0_{\theta}v)(a,\xi)=\frac12e^{2a+\frac{\theta\xi}{\kappa}}(1+\kappa\partial_\xi^2+\theta(2+\partial_a)\partial_\xi+\frac{\theta^2}{4\kappa}(4+4\partial_a+\partial_a^2))\caF_\ell( v)(a,\xi),
\end{equation*}
for a function $v \in \caS(\gR^2)$. To simplify this expression, let us also perform the change of variables
\begin{equation}
\varphi:(a,\xi)\mapsto (\xw,\yw):=\big(a+\frac{\theta}{2\kappa}\xi,a-\frac{\theta}{2\kappa}\xi\big).\label{eq-changevar}
\end{equation}
We obtain
\begin{equation}
\partial_t \thatp v_t(\xw,\yw)=\frac{i}{2\theta}e^{2\xw}\Big(1+\frac{\theta^2}{\kappa}(1+\partial_\xw)^2\Big)\thatp v_t(\xw,\yw).\label{eq-fouriersimple}
\end{equation}
if we denote $\thatp v_t(\xw,\yw)$ the image of the star-exponential $v_t=E_{\star^0_{\theta}}(t\lambda_F)$ by the above transformation $(\varphi^{-1})^*\circ\caF_\ell$. To solve this PDE, we choose a new unknown function $\thatp{u}_t(\xw,\yw)=e^{\xw-\yw}\thatp{v}_t(\xw,\yw)$, and the equation becomes:
\begin{equation*}
\partial_t \thatp u_t(\xw,\yw)=\frac{i}{2\theta}e^{2\xw}\Big(1+\frac{\theta^2}{\kappa}\partial_\xw^2\Big)\thatp u_t(\xw,\yw).
\end{equation*}
The corresponding general initial condition is $\thatp u_0(\xw,\yw)=e^{\xw-\yw}\thatp f(\xw,\yw)$, with $\thatp f:=(\varphi^{-1})^*\circ\caF_\ell(f)$ (or take $\thatp f(\xw,\yw)=\frac{2\pi\theta}{\kappa}\delta(\xw-\yw)$ to recover directly the star-exponential). If we change the variables once more $q:=\frac{\sqrt{2\kappa}}{\theta} e^{-\xw}$, we get
\begin{displaymath}
\partial_t \thatp u_t(q,\yw)=-\frac{i}{\theta}D \thatp u_t(q,\yw),\quad\text{ with }\quad D:=-\partial_q^2-\frac{1}{q}\partial_q-\frac{\kappa}{\theta^2 q^2}.
\end{displaymath}
The equation of eigenfunctions of $D$ associated to an eigenvalue $\lambda\in\gC$ turns out to be an adaptation of the Bessel equation of order $\frac{i\sqrt\kappa}{\theta}$:
\begin{equation}
\partial_q^2\varphi+\frac{1}{q}\partial_q\varphi+(\lambda+\frac{\kappa}{\theta^2q^2})\varphi=0.\label{eq-besseleqdiff}
\end{equation}
The general solution is known (see \cite[p. 97]{Watson:1966}) to be the following linear combination of Bessel functions of the first ($J$) and second ($Y$) kind:
\begin{equation*}
\varphi(q)=K_1 J_{\frac{i\sqrt\kappa}{\theta}}(\sqrt{\lambda}q)+K_2Y_{\frac{i\sqrt\kappa}{\theta}}(\sqrt{\lambda}q).
\end{equation*}

\subsection{Description of the functional transformation}
\label{subsec-funcdescr}

Before applying the initial condition to this solution, we want to have a closer look on the transformation $\thatp f:=(\varphi^{-1})^*\circ\caF_\ell(f)$ with $\caF_\ell$ given by \eqref{eq-partfourier} and $\varphi$ by \eqref{eq-changevar}. It has the form
\begin{equation}
\thatp f(\xw,\yw)=\int_\gR f(\frac12(\xw+\yw),\ell)e^{-i\frac{\kappa}{\theta}(\xw-\yw)\ell}\dd \ell,\label{eq-transtilde2}
\end{equation}
with new variables $(\xw,\yw):=\varphi(a,\xi)$ belonging also to $\gR^2$. The star-product $\star_\theta^0$, the complex conjugation and the scalar product can be transported by this functional transformation and we then obtain a (continuous) matrix Hilbert algebra.
\begin{proposition}
\label{prop-contmatrix}
Transported by \eqref{eq-transtilde2}, the Moyal product \eqref{eq-star0}, the complex conjugation, and the scalar product have the form of the standard matrix product, transpose-conjugation and scalar product on continuous matrices. Namely, we have for any $f_1,f_2\in L^2(\gR^2)$,
\begin{align}
& f_1\thatp\star_\theta f_2(\xw,\yw)=\frac{\kappa}{2\pi\theta}\int_\gR\ f_1(\xw,\eta)f_2(\eta,\yw)\dd\eta,\nonumber\\
& f^*(\xw,\yw)= \overline{f(\yw,\xw)},\label{eq-contmatrix2}\\
&\langle f_1,f_2\rangle=\frac{\kappa^2}{2\pi\theta}\int\ \overline{f_1(\xw,\yw)} f_2(\xw,\yw)\dd \xw\dd \yw.\nonumber
\end{align}
$(L^2(\gR^2),\thatp\star_\theta)$ is then a Hilbert algebra.
\end{proposition}
We note $\kM_{\thatp\star_\theta}(\gR^2)$ the corresponding bounded multiplier algebra. Then, the transformation \eqref{eq-transtilde2} extends as a spatial isomorphism of von Neumann algebras $\kM_{\star_\theta^0}(\gR^2)\to\kM_{\thatp\star_\theta}(\gR^2)$.
\begin{proof}
It is direct to obtain the expression of the product, the involution and the scalar product on $L^2(\gR^2)$. By definition, the transformation \eqref{eq-transtilde2} is a unitary *-isomorphism, so that $(L^2(\gR^2),\thatp\star_\theta)$ is a Hilbert algebra.
\end{proof}

\subsection{Orthogonality relation}

For a real order $\nu\in\gR$, there is a well-known orthogonality relation for Bessel functions given by 
\begin{equation}
\int_{\mathbb{R}^{+}} r J_{\nu}(kr)J_{\nu}(k^{'}r)dr = \frac{\delta (k-k^{'})}{k},
\end{equation}
understood in the distributional sense, for $k$ and $k'$ positive. However, this identity does not extend directly to the case of pure imaginary order $\nu$, and involving real Bessel functions. In this section, we will determine such an extension by using the method of Sturm-Liouville expansion \cite{Naimark:1968,Weidmann:1987,Titchmarsh:1962}. This orthogonality relation will be associated to the spectral measure of $D$ which will give directly the star-exponential.

To this aim, we consider real Bessel functions corresponding to a pure imaginary order.  See \cite{Dunster:1990a} for the definition:
\begin{equation*}
\tilde J_\nu(q):=\frac{1}{\sinh(\frac{\pi\nu}{2})}\Re(J_{i\nu}(q)),\qquad \tilde Y_\nu(q):=\frac{1}{\sinh(\frac{\pi\nu}{2})}\Re(Y_{i\nu}(q)).
\end{equation*}
By using the Liouville transformation $\varphi(q)\mapsto \sqrt{q}\varphi(q)$, the extension of Equation \eqref{eq-besseleqdiff} with an arbitrary spectral value $\lambda\in\gC$ becomes
\begin{equation}
\varphi''(q)+\Big(\lambda+\frac{1}{q^2}(\frac14+\frac{\kappa}{\theta^2})\Big)\varphi(q)=0.\label{eq-besseleqdiff2}
\end{equation}

We fix $\nu=\frac{\sqrt\kappa}{\theta}$ and consider the case where $\lambda=s^2\geq 0$. Then the real solutions of this equation are $\sqrt{q}\tilde J_\nu(sq)$ and $\sqrt{q}\tilde Y_\nu(sq)$. In the notations of \cite{Titchmarsh:1962}, we look at the following system $(\phi,\vartheta)$ of solutions with conditions
\begin{equation*}
\phi(a,s^2)=0,\quad \phi'(a,s^2)=-1,\qquad \vartheta(a,s^2)=1,\quad\vartheta'(a,s^2)=0,
\end{equation*}
for a fixed value of the parameter $a\in\gR_{>0}$. We find
\begin{align*}
&\phi(q,s^2)=\frac{\pi}{2}\sqrt{aq}\big(\tilde Y_\nu(sa)\tilde J_\nu(sq)-\tilde J_\nu(sa)\tilde Y_\nu(sq)\big),\\
&\vartheta(q,s^2)=\frac{\pi s}{2}\sqrt{aq}\big((\tilde Y'_\nu(sa)+\frac{\tilde Y_\nu(sa)}{2as})\tilde J_\nu(sq)-(\tilde J'_\nu(sa)+\frac{\tilde J_\nu(sa)}{2as})\tilde Y_\nu(sq)\big).
\end{align*}
In the terminology of the Sturm-Liouville theory, the bound $q=+\infty$ is a limit point case and
\begin{equation*}
m_2(s^2)=-\frac{1}{2a}-s\Big(\frac{\tilde J'_\nu(sa)+i\tilde Y'_\nu(sa)}{\tilde J_\nu(sa)+i\tilde Y_\nu(sa)}\Big)
\end{equation*}
is defined such that the continuation of $\vartheta(q,s^2)+m_2(s^2)\phi(q,s^2)$ to $\Im(s)>0$ is in $L^2(a,\infty)$ in the variable $s$. For this, use the asymptotic expansions
\begin{equation*}
\tilde J_\nu(q)=\sqrt{\frac{2}{\pi q}}\cos(q-\frac{\pi}{4})+O(q^{-\frac32}),\qquad \tilde Y_\nu(q)=\sqrt{\frac{2}{\pi q}}\sin(q-\frac{\pi}{4})+O(q^{-\frac32}).
\end{equation*}
For the bound $q=0^+$, we are in the limit circle case and 
\begin{equation*}
m_1(s^2)=-\frac{1}{2a}-s\Big(\frac{\tilde Y'_\nu(sa)-2s\nu \tilde J'_\nu(sa)}{\tilde Y_\nu(sa)-2s\nu \tilde J_\nu(sa)}\Big)
\end{equation*}
corresponds to a well-chosen point of this circle. For $s\geq 0$, $m_1(s^2)$ is real and we have
\begin{equation*}
-\Im\Big(\frac{1}{m_1(s^2)-m_2(s^2)}\Big)=\frac{\pi a}{2(1+4s^2\nu^2)}(\tilde Y_\nu(sa)-2s\nu \tilde J_\nu(sa))^2.
\end{equation*}

Consider now the other case for the eigenvalue $\lambda=-s^2<0$ ($s>0$), we can proceed in a similar way by using the solutions $\sqrt{q}\tilde I_\nu(sq)$ and $\sqrt{q}\tilde K_\nu(sq)$, with real functions \cite{Dunster:1990a}:
\begin{equation*}
\tilde I_\nu(q):=\Re(I_{i\nu}(q)),\qquad \tilde K_\nu(q):=K_{i\nu}(q).
\end{equation*}
We find explicitly
\begin{align*}
&\phi(q,-s^2)=\sqrt{aq}\big(\tilde K_\nu(sa)\tilde I_\nu(sq)-\tilde I_\nu(sa)\tilde K_\nu(sq)\big),\\
&\vartheta(q,-s^2)=s\sqrt{aq}\big((\tilde K'_\nu(sa)+\frac{\tilde K_\nu(sa)}{2as})\tilde I_\nu(sq)-(\tilde I'_\nu(sa)+\frac{\tilde I_\nu(sa)}{2as})\tilde K_\nu(sq)\big).
\end{align*}
Then,
\begin{equation*}
m_2(-s^2)=-\frac{1}{2a}-\frac{s\tilde K'_\nu(sa)}{\tilde K_\nu(sa)},\qquad m_1(-s^2)=-\frac{1}{2a}-s\Big(\frac{\tilde I'_\nu(sa)+\frac{\sinh(\pi\nu)}{2\pi\nu} \tilde K'_\nu(sa)}{\tilde I_\nu(sa)+\frac{\sinh(\pi\nu)}{2\pi\nu} \tilde K_\nu(sa)}\Big)
\end{equation*}
so that $-\Im\Big(\frac{1}{m_1(-s^2)-m_2(-s^2)}\Big)=0$.

We are now in position to apply Formula (3.1.12) of \cite[page 53]{Titchmarsh:1962} and Formula (3.1.5) of \cite[page 51]{Titchmarsh:1962} to determine the spectral measure of the differential operator involved in Equation \eqref{eq-besseleqdiff2}. We obtain the measure $\frac{\pi a}{2(1+4s^2\nu^2)}(\tilde Y_\nu(sa)-2s\nu \tilde J_\nu(sa))^2\dd\lambda$ for positive spectral values $\lambda=s^2$ and the measure $0$ for negative spectral values. To summarize, we obtained the following result.
\begin{theorem}
\label{thm-orth}
For positive $q,q'$ and real $\nu$, we have the following orthogonality relation
\begin{equation}
\int_0^\infty\big(\tilde Y_\nu(sq)-2s\nu \tilde J_\nu(sq)\big)\big(\tilde Y_\nu(sq')-2s\nu \tilde J_\nu(sq')\big)\frac{s}{1+4s^2\nu^2}\dd s=\frac{1}{\sqrt{qq'}}\delta(q-q'),\label{eq-orthbessel}
\end{equation}
understood in the distributional sense.
\end{theorem}

\subsection{Construction of the star-exponential}
\label{subsec-unique}

\begin{remark}
The operator $D=-\partial_q^2-\frac{1}{q}\partial_q-\frac{\kappa}{\theta^2 q^2}$ defined on $\caD(\gR)$ admits selfadjoint extensions, since its defect indices are equal by Sturm-Liouville theory. Such a selfadjoint extension, also denoted by $D$, can be decomposed with its spectral measure.
\end{remark}
Let us determine the spectral measure of $D$. We saw in Theorem \ref{thm-orth} what is the spectral measure of the operator $-\partial_q^2-\frac{1}{q^2}(\frac14+\frac{\kappa}{\theta^2})$: the orthogonality relation indeed corresponds to the resolution of the identity of this operator. So we deduce that the kernel of this operator is given by the same spectral integral, but with multiplication by the eigenvalue $s^2$:
\begin{equation*}
K'(q,q')=\int_0^\infty s^2\sqrt{qq'}A_{\frac{\sqrt\kappa}{\theta}}(q,s)A_{\frac{\sqrt\kappa}{\theta}}(q',s)\frac{s}{1+\frac{4\kappa s^2}{\theta^2}}\dd s
\end{equation*}
where we set $A_{\frac{\sqrt\kappa}{\theta}}(q,s):=\tilde Y_{\frac{\sqrt\kappa}{\theta}}(sq)-2s \frac{\sqrt\kappa}{\theta} \tilde J_{\frac{\sqrt\kappa}{\theta}}(sq)$. To get the operator $D$, we perform the inverse Liouville transformation, which consists here to intertwin this operator by the multiplication operator $\frac{1}{\sqrt q}$; so that the kernel of $D$ is
\begin{equation*}
K_D(q,q')=q'\int_0^\infty A_{\frac{\sqrt\kappa}{\theta}}(q,s)A_{\frac{\sqrt\kappa}{\theta}}(q',s)\frac{s^3}{1+\frac{4\kappa s^2}{\theta^2}}\dd s.
\end{equation*}
Note that for any measurable function $f$ on $\gR_+$, we have the explicit expression of the functional calculus $f(D)$, with kernel
\begin{equation*}
K_{f(D)}(q,q')=q'\int_0^\infty A_{\frac{\sqrt\kappa}{\theta}}(q,s)A_{\frac{\sqrt\kappa}{\theta}}(q',s)\frac{s f(s^2)}{1+\frac{4\kappa s^2}{\theta^2}}\dd s.
\end{equation*}
Since $\thatp u_t(q,\yw)=e^{-\frac{i}{\theta}tD}\thatp u_0(q,\yw)$, we obtain
\begin{equation*}
\thatp u_t(q,\yw)=\int_0^\infty q'\int_0^\infty e^{-\frac{i}{\theta}ts^2} A_{\frac{\sqrt\kappa}{\theta}}(q,s)A_{\frac{\sqrt\kappa}{\theta}}(q',s)\thatp u_0(q',\yw)\frac{s}{1+\frac{4\kappa s^2}{\theta^2}}\dd s\dd q'.
\end{equation*}
We change the variables $q=\frac{\sqrt{2\kappa}}{\theta} e^{-\xw}$ and $q'=\frac{\sqrt{2\kappa}}{\theta} e^{-\eta}$, and the function $\thatp v_t(\xw,\yw)=e^{-\xw+\yw}\thatp u_t(\xw,\yw)$:
\begin{equation}
\thatp v_t(\xw,\yw)=\frac{2\kappa}{\theta^2}\int_\gR \int_0^\infty e^{-\frac{i}{\theta}ts^2} e^{-\xw-\eta} A_{\frac{\sqrt\kappa}{\theta}}(\frac{\sqrt{2\kappa}}{\theta}e^{-\xw},s)A_{\frac{\sqrt\kappa}{\theta}}(\frac{\sqrt{2\kappa}}{\theta}e^{-\eta},s)\thatp f(\eta,\yw)\frac{s}{1+\frac{4\kappa s^2}{\theta^2}}\dd s\dd \eta.\label{eq-solution}
\end{equation}

The next result summarizes this section by giving the explicit expression of the non-formal star-exponential of $\lambda_F$.
\begin{theorem}
\label{thm-starexpbessel}
After transformation \eqref{eq-transtilde2}, the star-exponential of $\lambda_F$ is the multiplier associated to the function
\begin{equation*}
E_{\thatp\star_\theta}(\frac{i}{\theta}t\lambda_F)(\xw,\yw)=\frac{4\pi}{\theta}e^{-\xw-\yw}\int_{\gR_+} \,e^{-\frac{i}{\theta}s^2t} A_{\frac{\sqrt\kappa}{\theta}}(\frac{\sqrt{2\kappa}}{\theta}e^{-\xw},s) A_{\frac{\sqrt\kappa}{\theta}}(\frac{\sqrt{2\kappa}}{\theta}e^{-\yw},s)\frac{s}{1+\frac{4\kappa s^2}{\theta^2}}\dd s.
\end{equation*}
\end{theorem}
\begin{proof}
We saw in Proposition \ref{prop-contmatrix} that $(L^2(\gR^2),\thatp\star_\theta)$ is a Hilbert algebra isomorphic to $(L^2(\gR^2),\star_\theta^0)$ where the star-exponential of $\lambda_F$ exists and is well-defined (see section \ref{subsec-starprod}). So, the star-exponential $E_{\thatp\star_\theta}(\frac{i}{\theta}t\lambda_F)$ is also well-defined as a bounded multiplier and it satisfies Equation \eqref{eq-fouriersimple}. More precisely, the left multiplication by the star-exponential $E_{\thatp\star_\theta}(\frac{i}{\theta}t\lambda_F)\thatp \star_\theta \thatp f$ satisfies Equation \eqref{eq-fouriersimple} with initial condition $\thatp f$. 

Since $e^{-\frac{i}{\theta}tD}$ is a unitary operator, we see directly that \eqref{eq-solution} is a bounded solution of \eqref{eq-fouriersimple}: $\norm \thatp v_t\norm=\norm \thatp f\norm$. And it is straightforward to show that it is also the left part of a multiplier (see $L$ in \eqref{eq-multiplier2}): $\thatp v_t(\thatp f_1\thatp\star_\theta\thatp f_2)=\thatp v_t(\thatp f_1)\thatp\star_\theta\thatp f_2$.
By uniqueness of the star-exponential, we deduce that \eqref{eq-solution} corresponds exactly to 
$E_{\thatp\star_\theta}(\frac{i}{\theta}t\lambda_F)\thatp \star_\theta \thatp f$.
\end{proof}

This star-exponential, or more precisely its image by the intetwining operator from $\hat\star_\theta$ to $\star_\theta^1$, then corresponds to the geometric information that we wanted to compute for the moment $\lambda_F$, and to realize on the (local) contraction of $AdS_2$ for its natural star-product $\star_\theta^1$. We want now to express its link with the principal series representations.

\section{Link with principal series representations}
\label{sec-principal}

An application of \cite{Cahen:1996st} to the group $SL_2(\gR)$ gives the star-exponential of this group for another star-product in terms of the principal series representations. This other star-product actually coincides with the Moyal-Weyl product on another coordinate chart $\Psi_\kappa$ (see also \cite{Bieliavsky:1999st}). We present here this chart $\Psi_\kappa$, we give also another method as in \cite{Cahen:1996st} to obtain the expression of the star-exponential and we then show that it is not only a distribution but an element of the von Neumann algebra of the multipliers as in section \ref{subsec-unique}.

\subsection{Another coordinate chart}

To express the star-exponential of the $SL(2,\gR)$,  for a certain star-product on (a part) of its adjoint orbit, as the principal series representation, we need to deal with a coordinate system different from the one $\Phi_\kappa$ used  in section \ref{subsec-orbit}.

For $\kappa >0$, we consider $H_{\kappa}:= \sqrt{\kappa}H$ in the adjoint orbit $\mathcal{O}_{\kappa}$. Let us define the coordinate system $\Psi_{\kappa} : \gR^2 \rightarrow \mathcal{O}_{\kappa}$ by
\begin{equation*}
\Psi_\kappa(x E + y F):= \Ad_{e^{x E}e^{y F}}(H_{\kappa})=\sqrt{\kappa}(1+2x y)H - 2 \sqrt{\kappa} x (1+x y)E + 2 \sqrt{\kappa} yF.
\end{equation*}
$\Psi_{\kappa}$ is also a Darboux chart and its image corresponds to the whole Anti-deSitter space except one line. The KKS form is $\omega^\Psi:=\Psi_{\kappa}^*\omega_z^{\mathcal{O}_{\kappa}}=\sqrt{\kappa}\dd x\wedge\dd y$.
The image of $\Psi_\kappa$ is not invariant under the action of $SL(2,\gR)$ but it is for the action of the subgroup $\gS$. In this chart, this action of $\gS$ can be written
\begin{equation}
\Ad_{e^{aH}e^{\ell E}}\Psi_{\kappa}(xE+yF)=\Psi_{\kappa}((x+\ell)e^{2a}E+ye^{-2a}F).\label{eq-actsq}
\end{equation}

The image of $\Psi_\kappa$ contains exactly the one of $\Phi_\kappa$, the one of $\Phi'_\kappa$ and the line $\{\sqrt\kappa (H-2xE),\ x\in\gR\}$ (see Figure \ref{fig-AdS-habit}). One has in particular the following change of coordinates $j_{\kappa} := \Psi_{\kappa}^{-1} \circ \Phi_{\kappa}:\gR^2\to\gR\times\gR_+^*$ between the systems $\Phi_\kappa$ and $\Psi_\kappa$:
\begin{equation}
\label{eq-changej}
j_{\kappa}(aH+\ell E)=(\ell-\frac{1}{\sqrt{\kappa}})e^{2a}E+\frac{\sqrt{\kappa}}{2}e^{-2a}F.
\end{equation}

The moment maps $\lambda$ for the action of $G$ expressed in the coordinate system $\Psi_\kappa$ are
\begin{equation*}
\lambda_H(x,y)=\sqrt{\kappa}(1+2x y),\qquad \lambda_E(x,y)=\sqrt{\kappa}y,\qquad \lambda_F(x,y)=-\sqrt{\kappa}x (1+x y).
\end{equation*}

The part of the adjoint orbit corresponding to the coordinate system $\Psi_\kappa$ is symplectomorphic to $\gR^2$ with symplectic form $\sqrt\kappa\dd x\wedge\dd y$, so that we can consider the Moyal product. It has the expression: for $f_1,f_2\in\mathcal{S}(\gR^2)$,
\begin{equation}
\big(f_1\sharp_\theta f_2\big)(x,y):=\frac{\kappa}{\pi^2\theta^2}\int f_1(x_1,y_1) f_2(x_2,y_2) e^{\frac{-2i\sqrt{\kappa}}{\theta} (x_1y_2-x_2y_1+x_2y-xy_2+xy_1-x_1y)} \dd x_i\dd y_i\label{eq-sharp}
\end{equation}
Note that this Moyal product in this coordinate system is very different from \eqref{eq-star0} as you have to use the change of coordinates $j_\kappa$ \eqref{eq-changej}. The asymptotic expansion of \eqref{eq-sharp} writes
\begin{equation*}
\big(f_1\sharp_\theta f_2\big)(x,y)\sim \exp\Big(\frac{-i\theta}{2\sqrt{\kappa}}(\partial_{x_1}\partial_{y_2}-\partial_{y_1}\partial_{x_2})\Big)f_1(x_1,y_1)f_2(x_2,y_2)|_{x_1=x_2=x,\ y_1=y_2=y}.
\end{equation*}
The star-product $\sharp_\theta$ is also $\gS$-invariant and $\kg$-covariant, so that the star-exponential $E_{\sharp_\theta}(\frac{i}{\theta}\lambda_X)$ ($X\in\kg$) is well defined in the bounded multipliers $\kM_{\sharp_\theta}(\gR^2)$ of the Hilbert algebra $(L^2(\gR^2),\sharp_\theta)$.

\subsection{Resolution of the equation}
\label{subsec-starexppsi}

Let us find the explicit expression of the star-exponential $E_{\sharp_\theta}(\frac{i}{\theta}\lambda_X)$ as before. We will see that in this coordinate chart, the equations will be easier and one can solve them for arbitrary $X\in\kg$.

For $X=\alpha H +\beta E +\gamma F$ in $\kg$, the defining equation of the star-exponential of $\lambda_X$ is given by
\begin{equation}
\partial_t v_t(x,y)=\frac{i}{\theta}\lambda_X\sharp_{\theta} v_t(x,y).
\label{eq-defstarexp}
\end{equation}
As before, we take a general initial condition $v_0(x,y)=f(x,y)$, with $f\in L^2(\gR^2)$.

Let us use the same functional transformation $v\mapsto \that v$ \eqref{eq-transtilde2} as before to obtain a PDE. After the partial Fourier transformation in the variable $y$, we change the coordinates into
\begin{equation*}
\xz=x+\frac{\theta }{2\sqrt{\kappa}}\xi,\qquad \yz=x-\frac{\theta }{2\sqrt{\kappa}}\xi,
\end{equation*}
with $(\xz,\yz)\in\gR^2$ (note that there is a change of normalization with respect to \eqref{eq-transtilde2}). Then, we can compute partial Fourier transform of the left multiplication by moment maps:
\begin{itemize}
\item $\mathcal{F}_y\big(\lambda_H\sharp_\theta v\big)(x,\xi)=(\sqrt{\kappa}+i\theta)\that{v}  + 2i\theta \xz\partial_\xz \that{v}(\xz,\yz)$,
\item $\mathcal{F}_y\big(\lambda_E\sharp_\theta v\big)(x,\xi)=i\theta \partial_\xz \that{v}$, and 
\item $\mathcal{F}_y\big(\lambda_F\sharp_\theta v\big)(x,\xi)=-(\sqrt{\kappa}+i\theta)\xz\that{v} -i\theta\xz^{2}\partial_\xz\that{v}$.
\end{itemize}
Then, Equation \eqref{eq-defstarexp} becomes a PDE of order 1:
\begin{equation}
\partial_t \that v_t(\xz,\yz) =\frac{i}{\theta} \Big(  \alpha (\sqrt{\kappa} + i\theta) + 2i\alpha \theta \xz \partial_\xz + i \beta \theta \partial_\xz -\gamma (\sqrt{\kappa}+i\theta)\xz -i\theta\gamma \xz^2 \partial_\xz \Big) \that v_t(\xz,\yz).\label{eq-defstarexp2}
\end{equation}

Let us concentrate on the part of order one in the derivatives of this equation. It is given by
\begin{equation*}
\partial_t\that u_t(\xz,\yz)=\big(-2\alpha \xz -\beta + \gamma \xz^2 \big) \partial_\xz \that u_t(\xz,\yz),
\end{equation*}
together with the initial data $\that u_0(\xz,\yz)=\that{f}(\xz,\yz)$. We look for the integral curves $Y(t,\xz)$ of the vector field $(-\beta -2\alpha \xz + \gamma \xz^2)\partial_\xz$ with initial condition $Y(0)=\xz$; they satisfy the following equation:
\begin{equation*}
\frac{\dd Y}{\dd t}=-\beta-2\alpha Y+ \gamma Y^2.
\end{equation*}
The solutions are $Y(t,\xz)=M(t).\xz$ where the stereographic action of $SL(2,\gR)$ is given by
\begin{equation}
\begin{pmatrix} a& b \\ c & d \end{pmatrix}.\xz:=\frac{a\xz-c}{-b\xz+d},\label{eq-steract}
\end{equation}
$\tau:=\sqrt{\alpha^2+\beta\gamma}$ and the matrix
\begin{equation*}
M(t):=\begin{pmatrix} \cosh(\tau t)-\frac{\alpha}{\tau}\sinh(\tau t)& \frac{\gamma}{\tau}\sinh(\tau t) \\ \frac{\beta}{\tau}\sinh(\tau t) & \cosh(\tau t)+\frac{\alpha}{\tau}\sinh(\tau t) \end{pmatrix} =e^{-\alpha tH+\gamma tE+\beta tF},
\end{equation*}
with $X=\alpha H+\beta E+\gamma F$. So we have $M(t)=\sigma(e^{tX})$ where
\begin{equation*}
\sigma \begin{pmatrix} a& b \\ c & d \end{pmatrix}:=\begin{pmatrix} d& c \\ b & a \end{pmatrix}.
\end{equation*}
is a group automorphism and an involution. At the end, we obtain the solution of the order 1 part of the PDE as
\begin{equation*}
\that u_t(\xz,\yz)=\that f(M(t).\xz,\yz).
\end{equation*}

We come now to the general equation \eqref{eq-defstarexp2}. If we try the Ansatz $\that{v}_t=\rho_t \that u_t$ where $\that u_t$ was determined just above, and $\rho_t$ depends on the variable $\xz$ only and with the initial condition $\rho_0(\xz)\equiv 1$, we find the following equation for $\rho_t$:
\begin{equation*}
\partial_t\rho_t = \Big( (-\beta -2\alpha \xz + \gamma \xz^2 )\partial_\xz +\frac{i}{\theta}(\sqrt{\kappa} + i\theta)(\alpha -\gamma \xz)\Big) \rho_t (\xz).
\end{equation*}
Let us choose new coordinates:
\begin{equation*}
t' =t,\qquad \xz'=Y(t,\xz)=M(t).\xz.
\end{equation*}
Therefore, $\partial_{t'}=\partial_t+(\beta + 2\alpha \xz - \gamma \xz^2)\partial_\xz$. This gives the following equation on $\rho_t$, in the new coordinates
\begin{equation*}
\partial_{t'}\rho_{t'}(\xz')=\frac{i}{\theta}(\sqrt{\kappa}+ i\theta) (\alpha - \gamma M(t')^{-1}.\xz')\rho_{t'}(\xz'),
\end{equation*}
which admits as solution the function $\rho$ given by 
\begin{equation*}
\rho_{t'}(\xz')=\exp\Big( -(1-\frac{i\sqrt\kappa}{\theta})(\alpha t'-\gamma \int_0^{t'} M(s)^{-1}.\xz'\dd s))\Big).
\end{equation*}

This yields the following expression for the solution:
\begin{equation*}
\that v_t(\xz,\yz)=\exp\Big( -(1-\frac{i\sqrt\kappa}{\theta})(\alpha t-\gamma \int_0^{t} M(s).\xz\, \dd s) \Big) \that{f}(M(t).\xz,\yz).
\end{equation*}
A long but straightforward computation shows that
\begin{multline*}
\int_0^{t} M(s).\xz\dd s=\int_0^{t}\frac{(\cosh(\tau s)-\frac{\alpha}{\tau}\sinh(\tau s))\xz-\frac{\beta}{\tau}\sinh(\tau s)}{-\frac{\gamma}{\tau}\sinh(\tau s)\xz+\cosh(\tau s)+\frac{\alpha}{\tau}\sinh(\tau s)}\dd s\\
 =\frac{1}{\gamma}(\alpha t-\log\Big(\cosh(\tau t)+\frac{\alpha-\gamma \xz}{\tau}\sinh(\tau t)\Big)).
\end{multline*}
Then, the solution takes the form
\begin{equation*}
\that v_t(\xz,\yz)=\Big| \cosh(\tau t)+\frac{\alpha-\gamma \xz}{\tau}\sinh(\tau t)\Big|^{-(1-\frac{i\sqrt\kappa}{\theta})}\that{f}(M(t).\xz,\yz)=\Big((\caP^{+,i\mu}(M(t))\otimes\text{id})\that f\Big)(\xz,\yz)
\end{equation*}
where $\caP^{+,i\mu}$ denotes the principal series representation of the group $SL(2,\gR)$:
\begin{equation}
\mathcal{P}^{+,i\mu} \begin{pmatrix} a&b\\c&d \end{pmatrix} f(\xz) = \vert -b\xz+d\vert^{-1-i\mu} f(\frac{a\xz-c}{-b\xz+d}),\label{eq-princseries}
\end{equation}
the parameter is $\mu:=-\frac{\sqrt\kappa}{\theta}$ and the left (resp. right) part of the tensor product acts on the variable $\xz$ (resp. $\yz$). We obtained here a simple solution of the Equation \eqref{eq-defstarexp2} in terms of the principal series representation and of $M(t)=\sigma(e^{tX})$. Up to differences of notations, it coincides with the star-exponential of \cite{Cahen:1996st}.

\subsection{Star-exponential as principal series}

We denote by
\begin{equation*}
\that L_{tX}(\that f)(\xz,\yz)=\Big((\caP^{+,i\mu}(\sigma(e^{tX}))\otimes\text{id})\that f\Big)(\xz,\yz)
\end{equation*}
the previous solution to take into account the dependence in the initial condition $f\in L^2(\gR^2)$. Let us identify it with the left multiplication by the non-formal star-exponential $E_{\that\sharp_\theta}(\frac{i}{\theta}t\lambda_X)\that\sharp_\theta\that f$.

We can prove directly that this solution is a bounded multiplier.
\begin{lemma}
\label{lem-multstarexppsi}
For $X\in\kg$, the solution $\that L_{tX}$ corresponds to a unitary multiplier in $\kM_{\that\sharp_\theta}(\gR^2)$.
\end{lemma}
\begin{proof}
For any $\that f\in L^2(\gR^2)$ and fixed $t$, we write $M(t)=\begin{pmatrix} a&b\\c&d \end{pmatrix}$ and the unitarity condition is obtained by
\begin{multline*}
\norm \that L_{tX}(\that f)\norm^2=\frac{\kappa}{2\pi\theta}\int |\that L_X(\that f)(\xz,\yz)|^2\dd\xz\dd\yz =\frac{\kappa}{2\pi\theta}\int |-b\xz+d|^{-2}|\that f(M(t).\xz,\yz)|^2\dd\xz\dd\yz\\
 =\frac{\kappa}{2\pi\theta}\int |\that f(\xz',\yz)|^2\dd\xz'\dd\yz=\norm \that f\norm^2,
\end{multline*}
if we perform the change of variable $\xz'=M(t).\xz$ of Jacobian $|-b\xz+d|^{-2}$. To prove that $\that L_{tX}$ is the left part of a multiplier, just check that
\begin{equation*}
\that L_{tX}(\that f_1\that\sharp_\theta \that f_2)(\xz,\yz)=|-b\xz+d|^{-1+i\frac{\sqrt\kappa}{\theta}} \frac{\sqrt\kappa}{2\pi\theta}\int \that f_1(M(t).\xz,\eta)\that f_2(\eta,\yz)\dd\eta= \that L_{tX}(\that f_1) \that\sharp_\theta \that f_2(\xz,\yz)
\end{equation*}
\end{proof}

\begin{theorem}
For any $X\in\kg$, the solution $\that L_{tX}$ given in terms of the principal series representation coincides with the non-formal star-exponential $E_{\that\sharp_\theta}(\frac{i}{\theta}t\lambda_X)$.
\end{theorem}
\begin{proof}
There are two ways to prove this identification. First, one can proceed as in section \ref{subsec-unique}. The operators $(\sqrt{\kappa}+i\theta)+ 2i\theta \xz\partial_\xz$, $i\theta \partial_\xz$ and $-(\sqrt{\kappa}+i\theta)\xz -i\theta\xz^{2}\partial_\xz$ are symmetric, which proves the uniqueness of the solution of the defining equation. And it has to identify with the star-exponential also satisfying this equation.

We sketch the other method. By Lemma \ref{lem-multstarexppsi}, the solution $\that L_{tX}(\that f)$ is a unitary multiplier. In a similar way, one can show that $t\mapsto \that L_{tX}(\that f)$ is a strongly continuous one-parameter group valued in the unitary multipliers, so that the Stone theorem applies and this group is the exponential of a anti-selfadjoint operator. By using the defining equation, we conclude that this generator is $\frac{i}{\theta}\lambda_X$ and that this one-parameter group is the non-formal star-exponential.
\end{proof}

Note that the BCH property is a consequence of the fact that the principal series is a representation. Moreover, this fact implies also that the star-exponential can be defined at the level of the group $G=SL(2,\gR)$. Setting $\caE_{\that\sharp_\theta}(e^{tX}):=E_{\that\sharp_\theta}(t\lambda_X)$, we have the following result showing a better regularity than in \cite{Cahen:1996st}.
\begin{proposition}
\label{prop-bchpsi}
The star-exponential at the level of the group $G$ is a continuous map
\begin{equation*}
\caE_{\that\sharp_\theta}:G\to\kM_{\that\sharp_\theta}(\gR^2)
\end{equation*}
for the weak topology of $\kM_{\that\sharp_\theta}(\gR^2)$, as well as a group homomorphism.
\end{proposition}
\begin{proof}
First, due to the expressions of $\that L_{tX}$ in terms of the principal series, the star-exponential $\caE_{\that\sharp_\theta}(e^{tX}):=E_{\that\sharp_\theta}(t\lambda_X)$ induces a well-defined map $\caE_{\that\sharp_\theta}:G\to\kM_{\that\sharp_\theta}(\gR^2)$ that takes the form
\begin{equation}
\caE_{\that\sharp_\theta}(g)\that\sharp_\theta \that f= (\caP^{+,-i\frac{\sqrt\kappa}{\theta}}(\sigma(g))\otimes\text{id})\that f,\label{eq-expr1starexp}
\end{equation}
for any $g\in G$ and $\that f\in L^2(\gR^2)$. Since $\sigma$ is an automorphism and $\caP^+$ is a continuous representation (for the weak topology), we deduce that the left multiplication is continuous in the variable $g\in G$. Group homomorphism property is a translation of the BCH property at the level of the group.
\end{proof}

\section{Computing intertwining operators}
\label{sec-inter}

In section \ref{sec-direct}, we obtained the non-formal star-exponential of $\lambda_H$, $\lambda_E$ and $\lambda_F$ in the natural coordinate chart $\Phi_\kappa$ associated to the contraction of $G=SL(2,\gR)$, and for the Moyal star-product $\star_\theta^0$. In order to link it to the principal series, we saw in section \ref{sec-principal} the star-exponential of any generator $\lambda_X$ in the other coordinate chart $\Psi_\kappa$ and for the Moyal product $\sharp_\theta$ of this other coordinate chart, and we were able to obtain this star-exponential at the level of the group $G$. Now, we would like to compare these two expressions of the star-exponential.

To this aim, we have to find a (unitary) intertwining operator $W$ between the Moyal product $\sharp_\theta$ on $\Psi_\kappa$ and the Moyal product $\star_\theta^0$ (or equivalently the product $\star_\theta^1$ related by $T_{01}$) on $\Phi_\kappa$. We give here three different methods to find explicitly such an intertwiner. We expose these three methods because they are somehow general ways to obtain explicit intertwiners between star-products and we believe they can be used in various different contexts.

\subsection{Method via quantizations}
\label{subsec-retractqu}

Let us expose the first method that uses quantization map associated to the star-products. More precisely, the star-products $\sharp_\theta$ and $\star_\theta^1$ have associated quantization maps $\Omega$ and $\Omega^1$, which are both equivariant under a common symmetry $\gS$. Then, taking the inverse of one quantization map composed with the other quantization map (presented on the same Hilbert space) gives the intertwiner. This method comes from ideas of retract with shared symmetry. 

\subsubsection{Quantization of $\sharp_\theta$}

First, we want to express the quantization map associated to $\sharp_\theta$ on the chart defined by $\Psi_\kappa$. The Moyal product is of course associated to the Weyl quantization. However, in order to see the equivariance of this quantization with respect to a larger group of symmetry, let us proceed as below. Indeed, the shared symmetry $\gS$ will be part of the translation group and also part of symplectic group for which the Weyl quantization is also equivariant.

Let $\mathfrak{H}=\gR^2\times\gR$ be the Heisenberg group. We denote by $e_0,e_1$ the generators of the $\gR^2$-part and $Z$ the generator of its center. Let $\underline H,\underline E, \underline F$ the generators of $Sp(2,\gR)$ acting linearly (by matrix action) on $(\gR^2,\sqrt\kappa \dd x\wedge\dd y)$. We denote by $\underline\gS$ the subgroup of $Sp(2,\gR)\ltimes \mathfrak{H}$ generated by $\underline H$ and $e_0$, which is isomorphic to $\gS$, and by $\underline G$ the one generated by $\underline H,\underline E, e_0,e_1,Z$. We have the following relations:
\begin{align}
&[\underline{H},\underline{E}]=4\underline{E} & [\underline{H},\underline{F}]=-4\underline{F}\nonumber\\
&[\underline{H},e_0]=2e_0 &[\underline{H},e_1]=-2e_1 \nonumber\\
&[\underline{E},e_1]=e_0 &[e_0,e_1]=Z.\label{eq-comrelundg}
\end{align}

We denote $\tilde{\underline G}=\gR\underline H\ltimes\ehH$, it is a subgroup of $\underline G$. Let
\begin{equation*}
(a,q,v,t):=(\begin{pmatrix} e^{2a} &0 \\ 0& e^{-2a} \end{pmatrix}, qe_0+ve_1,tZ)
\end{equation*}
a coordinate system of $\tilde {\underline G}$ and the group law yields in this coordinate system:
\begin{equation*}
(a,q,v,t)\fois (a',q',v',t')=(a+a',e^{-2a'}q+q',e^{2a'}v+v',t+t'+\frac12(e^{-2a'}qv'-e^{2a'}vq')).
\end{equation*}
The coadjoint orbit $\tilde\caO_\kappa$ of this group associated to the form $\sqrt\kappa Z^\ast$ can be expressed as
\begin{equation*}
\Ad^\ast_{(a,q,v,t)}(\sqrt\kappa Z^\ast)=\sqrt\kappa\Big(2qv\underline H^\ast-ve^{-2a}e_0^\ast+qe^{2a}e_1^\ast+Z^\ast\Big).
\end{equation*}
We see that $(q'=e^{2a}q,v'=e^{-2a}v)$ forms a coordinate system of the two-dimensional space $\tilde\caO_\kappa$. The KKS form is $\sqrt\kappa \dd q'\wedge\dd v'$. The action of an element $(a,\ell):=e^{a\underline H}e^{\ell e_0}$ of $\underline\gS\subset\tilde {\underline G}$ on $\tilde\caO_\kappa$ can be read in this chart as $(a,\ell)\fois (q',v')=((\ell+q')e^{2a},v'e^{-2a})$, which is the same as \eqref{eq-actsq}, so that this coadjoint orbit $\tilde\caO_\kappa$ coincides $\gS$-equivariantly with the symplectic space $(\gR^2,\sqrt\kappa \dd x\wedge \dd y)$ corresponding to the chart $\Psi_\kappa$ of the adjoint orbit $\caO_\kappa$ of $G=SL(2,\gR)$.

We use Kirillov's orbits method to construct a representation of $\tilde {\underline G}$. $\kb:=\langle \underline H,e_1,Z\rangle$ is a polarization affiliated to this coadjoint orbit, and $\chi(a,0,v,t)=e^{i\frac{\sqrt\kappa}{\theta} Z^\ast(\log(a,0,v,t))}=e^{i\frac{\sqrt\kappa}{\theta} t}$ on the subgroup $B$ generated by $\underline H,e_1,Z$.
We denote $Q=\tilde G/B\simeq \gR e_0$. The induced representation comes from the left regular representation acting on $B$-equivariant smooth functions. Since the measure $\dd q$ is not $\tilde {\underline G}$-invariant, the representation has to be corrected by some weight to be unitary. It is given by $U:\tilde{\underline G}\to \caL(L^2(Q))$ with
\begin{equation}
U(a,q,v,t)\varphi(q_0)=e^{-a}e^{i\frac{\sqrt\kappa}{\theta}\Big(t+\frac12qv-e^{-2a}q_0v\Big)}\varphi(e^{-2a}q_0-q).\label{eq-schro}
\end{equation}
We can reduce it to the group $\underline\gS$: $U(a,q,0,0)\varphi(q_0)=e^{-a}\varphi(e^{-2a}q_0-q)$. One can also introduce an involution on $\tilde {\underline G}$ by
\begin{equation*}
\sigma(a,q,v,t)=(a,-q,-v,t)
\end{equation*}
which is compatible with the polarization. Then, the Weyl-type quantizer \cite{Bieliavsky:2010kg} is a map $\Omega:\tilde {\underline G}\to\caL(L^2(Q))$ defined by
\begin{equation*}
\Omega(a,q,v,t)\varphi(q_0)=U(a,q,v,t)\sigma^\ast U(a,q,v,t)^{-1}\varphi(q_0)= e^{\frac{2i\sqrt\kappa}{\theta}(qv-e^{-2a}q_0v)}\varphi(2e^{2a}q-q_0),
\end{equation*}
which is $\tilde {\underline G}$-equivariant (and therefore $\underline\gS$-equivariant). We can notice that this map is actually well-defined on the coadjoint orbit $\tilde\caO_\kappa$ (take coordinates $(q'=e^{2a}q,v'=e^{-2a}v)$). As expected, the quantization map $\Omega:L^1(\tilde\caO_\kappa)\to\caL(L^2(Q))$ defined by
\begin{equation}
\Omega(f)\varphi(q_0):=\frac{\sqrt\kappa}{\pi\theta}\int_{\tilde\caO_\kappa} f(q',v') \Omega(q',v')\dd q'\dd v'\label{eq-quweyl}
\end{equation}
coincides exactly with the Weyl quantization. We know in particular that $\tr(\Omega(f)\Omega(a,q,v,t))=\frac12 f(e^{2a}q,e^{-2a}v)$ and that $\Omega(f_1\sharp_\theta f_2)=\Omega(f_1)\Omega(f_2)$.

\subsubsection{Quantization of $\star^1_{\theta}$}

We recall here results from \cite{Bieliavsky:2010kg}. There are two inequivalent irreducible induced representations in the unitary dual of $\gS$. They can be obtained by the method of coadjoint orbits due to Kirillov, with $\ee=\pm1$:
\begin{equation}
U_\ee^1(a,\ell)\varphi(a_0)= e^{\frac{i\kappa\ee}{2\theta}e^{2(a-a_0)}\ell}\varphi(a_0-a),\label{eq-unitindrep}
\end{equation}
for $(a,\ell)\in \gS$, $\varphi\in L^2(A_\ee)$ and $a_0\in A_\ee$, where we denote by $A_\ee$ the subgroup generated by $H$ of Lie algebra $\ka_\ee$, and the sign $\ee$ on $A_\ee$ is just an indication of the chosen representation. These representations $U_\ee^1:\gS\to\caL(L^2(A_\ee))$ are unitary and irreducible. A weight $\bfm$ is a function on $A_\ee$. There is a particular weight:
\begin{equation}
\bfm_0(a)=2\sqrt{\cosh(2a)}.\label{eq-multiplier}
\end{equation}
The symmetric structure of $\gS$ comes from the involutive automorphism $\sigma$ which can be restricted to $A_\ee$:
\begin{equation}
\sigma^\ast\varphi(a)=\varphi(-a).\label{eq-sigma}
\end{equation}
Then, for $\bfm$ a weight with $\bfm(0)=1$, the Weyl-type quantizer is given by
\begin{multline}
\Omega^1_{\ee}(a,\ell)\varphi(a_0):=U_\ee^1(a,\ell)\bfm\bfm_0\sigma^\ast U_\ee^1(a,\ell)^{-1}\varphi(a_0)\\
=2\sqrt{\cosh(2(a-a_0))} \bfm(a-a0) e^{\frac{i\kappa\ee}{\theta}\sinh(2(a-a_0))\ell}\varphi(2a-a_0).\label{eq-qumap}
\end{multline}

On smooth functions with compact support, one has the quantization map $\Omega^1_{\ee}:\caD(\gS)\to\caL(L^2(A_\ee))$
\begin{equation*}
\Omega^1_{\ee}(f):=\frac{\kappa}{2\pi\theta}\int_\gS f(a,\ell)\Omega_{\ee}(a,\ell)\dd a\dd \ell.
\end{equation*}
The normalization was chosen such that $\Omega^1_{\ee}(1)=\gone$. Moreover, it is $\gS$-equivariant, that is
\begin{equation*}
\forall (a,\ell)\in \gS\quad:\quad \Omega^1_{\ee}((a,\ell)^* f)= U_\ee^1(a,\ell) \Omega^1_{\ee}(f) U_\ee^1(a,\ell)^{-1}.
\end{equation*}

Moreover, the unitary representation $U_\ee^1:\gS\to\caL(L^2(A_\ee))$ induces a resolution of the identity. Indeed, by denoting the norm $\norm\varphi\norm^2_w:=\int |\varphi(a)|^2e^{2a}\dd a$ and $\varphi_{(a,\ell)}(a_0)=U_\ee^1(a,\ell)\varphi(a_0)$ for $(a,\ell)\in\gS$ and $\varphi\in L^2(A_\ee)$ (such that this norm exists and does not vanish), we have
\begin{equation*}
\frac{\kappa}{2\pi\theta\norm\varphi\norm^2_w}\int_\gS |\varphi_{(a,\ell)}\rangle\langle\varphi_{(a,\ell)}|\dd a\dd \ell=\gone.
\end{equation*}
This resolution of identity shows that the trace has the form
\begin{equation}
\tr(T)=\frac{\kappa}{2\pi\theta\norm\varphi\norm^2_w}\int_\gS \langle\varphi_{(a,\ell)}, T\varphi_{(a,\ell)}\rangle\dd a\dd \ell\label{eq-tracesymb}
\end{equation}
for $T\in\caL^1(L^2(A_\ee))$. As for the Weyl quantization, there is a left-inverse of the quantization map $\Omega_{\ee}^1$ given by the formula
\begin{equation*}
\forall f\in L^1(\gS),\ \forall (a,\ell)\in\gS\quad:\quad \tr(\Omega_{\ee}^1(f) \Omega_{\ee}^1(a,\ell))=f(a,\ell),
\end{equation*}
if the weight is unitary: $|\bfm(a)|=1$. This quantization is compatible with the star-product: $\Omega_{\ee}(f_1\star_{\ee\theta,\bfm}^1 f_2)=\Omega_{\ee}^1(f_1) \Omega_{\ee}^1(f_2)$, where the associative star-product $\star_{\ee\theta,\bfm}^1$ corresponds to \eqref{eq-star1} by changing the parameter to $\ee\theta$ and also adding $\frac{\bfm(a_1-a)\bfm(a-a_2)}{\bfm(a_1-a_2)}$ inside the integral of \eqref{eq-star1}.

\subsubsection{Intertwining operator}
\label{subsubsec-intqu}

Let us exhibit first an intertwining operator between the representations $U$ \eqref{eq-schro} restricted to $\underline\gS\simeq\gS$ and $U_\ee^1$ \eqref{eq-unitindrep} of the group $\gS$. In the spirit of the retract method, we define
\begin{equation*}
J_\ee:=\int_\gS |U(a,\ell)\varphi_0\rangle\langle U_\ee(a,\ell)\varphi_1| \dd a\dd \ell
\end{equation*}
for adapted $\varphi_0\in L^2(Q)$ and $\varphi_1\in L^2(A_\ee)$. A direct computation gives:
\begin{equation*}
J_\ee\psi(q_0)=N_\kappa \int_\gR e^{-\frac{i\kappa\ee}{2\theta}e^{-2a_0}q_0}e^{-a_0}\psi(a_0)\dd a_0
\end{equation*}
with $N_\kappa=\int_\gS e^{a}e^{\frac{i\kappa\ee}{2\theta}e^{-2a}\ell}\overline{\underline\varphi_0(a)}\varphi_0(\ell)\dd a\dd \ell$. We arrange the choice of $\varphi_0\in L^2(Q)$ and $\varphi_1\in L^2(A_\ee)$ such that $N_\kappa=\sqrt{\frac{\kappa}{2\pi\theta}}$. Then, we have

\begin{proposition}
\label{prop-j}
The operator $J:=J_++ J_-: L^2(A_+)\oplus L^2(A_-)\to L^2(Q)$ of inverse $J^*=J_+^*\oplus J_-^*$, defined by
\begin{equation*}
J(\psi_+,\psi_-)(q_0)=J_ +(\psi_+)(q_0)+J_-(\psi_-)(q_0),\qquad J^*(\varphi)(a_+,a_-)=(J_+^*(\varphi)(a_+),J_-^*(\varphi)(a_-)),
\end{equation*}
is unitary and intertwines the representations $U$ and $U_+^1\oplus U_-^1$:
\begin{equation*}
\forall (a,\ell)\in\gS\quad :\quad J\big(U_+^1(a,\ell)\oplus U_-^1(a,\ell)\big)=U(a,\ell)J.
\end{equation*}
\end{proposition}
\begin{proof}
Indeed, the expression of the adjoints is
\begin{equation*}
\forall \varphi\in L^2(Q)\quad :\quad J_\ee^\ast\varphi(a_0)=\sqrt{\frac{\kappa}{2\pi\theta}}e^{-a_0}\int_\gR e^{\frac{i\kappa\ee}{2\theta}e^{-2a_0}q_0}\varphi(q_0)\dd q_0.
\end{equation*}
Then, it is straightforward to show $J_\ee^*J_\ee=\gone$, $J_\ee^*J_{-\ee}=0$ and $J_+J_+^*+J_-J_-^*=\gone$.
\end{proof}

We want to consider the Weyl quantization $\Omega$ \eqref{eq-quweyl} but on the Hilbert space $L^2(A_+)\oplus L^2(A_-)$. Therefore, we define $\tilde\Omega(f):= J^{-1}\Omega(f) J$ and adopt a matrix notation $\tilde\Omega(f)_{\ee\ee'}$. The computation of its expression gives:
\begin{equation*}
\tilde\Omega(f)_{\ee\ee'}\varphi(a_0)=\frac{\kappa}{2\pi\theta}e^{-a_0}\int f\Big(q,\frac{\sqrt\kappa}{4}(\ee e^{-2a_0}+\ee' e^{-2a})\Big)e^{\frac{i\kappa}{2\theta}q(\ee e^{-2a_0}-\ee' e^{-2a})}e^{-a}\varphi(a)\dd a\dd q.
\end{equation*}

Finally, we define
\begin{equation*}
W_\ee(f)(a,\ell):=\tr(\tilde\Omega(f)_{\ee\ee}\Omega_{\ee}^1(a,\ell)).
\end{equation*}
A straightforward computation permits to obtain
\begin{multline*}
W_\ee(f)(a,\ell)=\frac{\kappa}{\pi\theta}e^{-2a}\int_{\gR^2} \sqrt{\cosh(2(a_0-a))}\bfm(a_0-a) e^{\frac{i\kappa\ee}{\theta}\sinh(2(a_0-a))(\ell-e^{-2a}q)}\\
f\Big(q,\frac{\sqrt\kappa\ee}{2}e^{-2a}\cosh(2(a_0-a))\Big)\dd a_0\dd q
\end{multline*}
or with the change of variables $\eta=\sinh(2(a_0-a))$,
\begin{equation}
W_\ee(f)(a,\ell)=\frac{\kappa}{2\pi\theta}e^{-2a}\int_{\gR^2} (1+\eta^2)^{-\frac14}\bfm\big(\frac12\text{Arcsinh}(\eta)\big)e^{\frac{i\kappa\ee}{\theta}\eta(\ell-e^{-2a}q)}f\Big(q,\frac{\sqrt\kappa\ee}{2}e^{-2a}\sqrt{1+\eta^2}\Big)\dd \eta\dd q.\label{eq-intertw}
\end{equation}
\begin{theorem}
\label{thm-wqu}
For a unitary weight $\bfm$, the operator $W_\ee$ is an intertwiner between the two star-products:
\begin{equation*}
W_\ee(f_1 \sharp_\theta f_2)=W_\ee(f_1)\star_{\ee\theta,\bfm}^1 W_\ee(f_2),
\end{equation*}
which is $\gS$-equivariant. Moreover, $W:=W_+\oplus W_-: L^2(\gR^2)\to L^2(\gS)\oplus L^2(\gS)$ is unitary.
\end{theorem}
\begin{proof}
Indeed, we have $W_\ee(f_1 \sharp_\theta f_2)(a,\ell) =\tr(\tilde\Omega(f_1)\tilde\Omega(f_2)\Omega_{\ee}^1(a,\ell))$ since $\Omega(f_1 \sharp_\theta f_2)=\Omega(f_1)\Omega(f_2)$. In the same way, we have
\begin{equation*}
W_\ee(f_1)\star_{\ee\theta,\bfm}^1 W_\ee(f_2)=\tr(\Omega_{\ee}^1(W_\ee(f_1))\Omega_{\ee}^1(W_\ee(f_2))\Omega_{\ee}^1(a,\ell))
\end{equation*}
By using the following property, straightforward to check,
\begin{equation*}
\Omega_{\ee}^1(\tr(T\Omega_{\ee}^1(\fois,\fois)))=T,
\end{equation*}
we obtain that $\Omega_{\ee}^1(W_\ee(f))=\tilde\Omega(f)$, which permits to show that $W_\ee$ is an intertwiner. Since $\Omega$, resp. $\Omega_\ee^1$, is a unitary operator from $L^2(\gR^2)$, resp. $L^2(\gS)$, onto Hilbert-Schmidt operators $\caL^2(L^2(Q))$, resp. $\caL^2(L^2(A_\ee))$, and since $J$ is unitary (Proposition \ref{prop-j}), we obtain directly the unitarity of $W$.

Moreover, by denoting $\tau$ the action \eqref{eq-actsq} of $\gS$ on $\gR^2$: $\tau_{(a,\ell)}(x,y)=((x+\ell)e^{2a},ye^{-2a})$, we have
\begin{equation*}
W(\tau_{(a,\ell)}^\ast f)(a_0,\ell_0)=W(f)(a_0+a,\ell_0+\ell e^{-2a_0})=W(f)((a,\ell)\fois (a_0,\ell_0)).
\end{equation*}
\end{proof}

We see here and from Proposition \ref{prop-j} that it is essential for unitarity to take two copies of $L^2(\gS)$ in the range of the intertwiner $W$. And unitarity is necessary for relating Hilbert algebras or their multipiers. Such a unitary intertwiner preserves all the functional properties of the star-exponential when acting on.

\subsection{Geometric method via equations}
\label{subsec-retractgeo}

In this section, we expose another method, also based on retract ideas of shared symmetries, but using geometric considerations and PDE instead of quantization maps. So, this method can be used for star-products even if there is no quantization map available. Let us consider here only formal star-products. The basic idea to find a $\gS$-equivariant intertwiner between the star-products $\sharp_\theta$ and $\star_\theta^1$ is to notice the following result, in the notations of section \ref{subsec-retractqu}.
\begin{lemma}
\label{lem-moyal}
The formal version of the product $\sharp_\theta$ \eqref{eq-sharp} is the unique star-product on $(\gR^2,\sqrt\kappa \dd x\wedge\dd y)$ strongly invariant under $\underline G$, the group generated by $\underline H$, $\underline E$, $e_0$, $e_1$ and $Z$.
\end{lemma}
This means that this Moyal product is completely characterized by the action of $\underline G$, there is no need to consider the action of the generator $\underline F$ (which is quite complicated) in the following. For $X\in\underline\kg$, we note $X^*$ the associated fundamental vector field on $\gR^2$. Strong invariance of $\sharp_\theta$ means that
\begin{equation*}
\forall X\in\underline\kg\quad:\quad [\underline\lambda_X,f]_{\sharp_\theta}=-i\theta X^*f
\end{equation*}
with $\underline\lambda$ the moment map of the affine action of $\underline G$ on $\gR^2$. A $\gS$-equivariant intertwiner $W$ would then leave invariant (or just change the coordinates with $j_\kappa$) the vector fields $\underline H^*$, $e_0^*$ corresponding to the $\underline\gS$-part in $\underline G$, but $W$ will transform $\underline E^*$, $e_1^*$ and $Z^*$ into $\star_\theta^1$-derivations. But such derivations can be classified and this gives a strong constraint on $W$ that can be re-expressed by a PDE. Solving this PDE produces the possible intertwiners $W$. 

\subsubsection{Equation on the intertwiner}

Let us determine all Lie algebra homomorphisms $\underline\kg\to\Der(\ks,\star_\theta^{1})$, with conditions on $\underline\ks$.
\begin{proposition}
The set of Lie algebra morphisms $D:\underline\kg\to\Der(\ks,\star_\theta^{1})$ satisfying $\forall X\in\underline\ks$ $D_X=(j^{-1}_\kappa)_*X^*$ is a complex two-dimensional manifold.
\end{proposition}
\begin{proof}
Since the de Rham cohomology of $\gS$ is trivial, we deduce that all the derivations of the formal star-product $\star_\theta^1$ are inner. For $X\in\underline\kg$, set $D_X=:\frac{i}{\theta}[\Lambda_X,\fois]_{\star_\theta^1}$ where $\Lambda_X$ is defined up to a constant. $D$ is a Lie algebra morphism, so due to Jacobi identity, we obtain that
\begin{equation*}
\forall X,Y\in\underline\kg\quad:\quad\Lambda_{[X,Y]}=\frac{i}{\theta}[\Lambda_X,\Lambda_Y]_{\star_\theta^1}
\end{equation*}
up to a constant term. For $X\in\underline\ks$ and $Y\in\underline\kg$, we have $((j^{-1}_\kappa)_*X^*)(\Lambda_Y)=\Lambda_{[X,Y]}$ up to a constant. These equations, which are due to the shared symmetry $\gS$, are sufficient to obtain the expression of $D$. First, note that $(j^{-1}_\kappa)_*\underline H^*=-\partial_a$ and $(j^{-1}_\kappa)_* e_0^*=-e^{-2a}\partial_\ell$ which actually coincide with the expressions of $H^*$ and $E^*$ for the action of $\gS$ on $\caO_\kappa$ in the $\Phi_\kappa$-coordinates. Therefore, we have $\Lambda_{\underline H}=\lambda_H$ and $\Lambda_{e_0}=\lambda_E$. Let us now write these equations, for $X=\underline H,e_0$ and $Y=\underline E,e_1,Z$ with the help of \eqref{eq-comrelundg}:
\begin{align*}
&-\partial_a\Lambda_{\underline E}=4\Lambda_{\underline E}+k_1,\qquad -\partial_a\Lambda_{e_1}=-2\Lambda_{e_1}+k_2,\qquad -\partial_a\Lambda_Z=k_3\\
&-e^{-2a}\partial_\ell\Lambda_{\underline E}=k_4,\qquad -e^{-2a}\partial_\ell\Lambda_{e_1}=\Lambda_Z+k_5,\qquad -e^{-2a}\partial_\ell\Lambda_{Z}=k_6,
\end{align*}
where $k_i$ are undetermined complex constants. The solutions are
\begin{equation*}
\Lambda_{\underline E}=\alpha e^{-4a},\qquad \Lambda_{e_1}=(\beta-\gamma \ell)e^{2a},\qquad \Lambda_Z=0
\end{equation*}
with $\alpha,\beta,\gamma$ complex constants, and up to constant terms. The condition $\frac{i}{\theta}[\Lambda_{\underline E},\Lambda_{e_1}]_{\star_\theta^1}=\Lambda_{e_0}$ (see \eqref{eq-comrelundg}), up to a constant term, implies $\gamma=\frac{\kappa^2}{8\alpha}$. We have thus two parameters $\alpha,\beta$ to parametrize the set of morphisms $D^{(\alpha,\beta)}$ of this Proposition:
\begin{align*}
\Lambda_{\underline H}=\kappa \ell,\qquad \Lambda_{\underline E}=\alpha e^{-4a},\qquad \Lambda_{e_0}=\frac{\kappa}{2}e^{-2a},\qquad  \Lambda_{e_1}=(\beta-\frac{\kappa^2}{8\alpha} \ell)e^{2a},\qquad \Lambda_Z=0.
\end{align*}
\end{proof}

Suppose that there exists an intertwiner $W$ between $\sharp_\theta$ and $\star_\theta^1$. We set $\tilde D^{(\alpha,\beta)}_X:=W^{-1}D^{(\alpha,\beta)}_XW$ for $X\in\underline\kg$. $\tilde D^{(\alpha,\beta)}$ describe the Lie algebra homomorphisms $\underline\kg\to\Der(\kq,\sharp_\theta)$ that satisfy $\forall X\in\underline\ks$, $\tilde D_X=X^*$. But we saw below Lemma \ref{lem-moyal} that the fundamental vector fields $X^*$ (due to the strong invariance) give such a Lie algebra homomorphism. So there exist $\alpha,\beta\in\gC$ such that $\forall X\in\underline\kg$,
\begin{equation}
W^{-1}D^{(\alpha,\beta)}_XW=\tilde D^{(\alpha,\beta)}_X=X^*.\label{eq-retractgeo1}
\end{equation}
For $X\in\underline\ks$ (elements of the shared symmetry), this is a tautology. But this equation evaluated on the other generators will permit to determine the intertwiner $W$.

The Schwartz kernel lemma together with the $\gS$-equivariance of $W$ lead us to the following form of the intertwiner
\begin{equation*}
(W\varphi)(a,\ell)=\int_{\gR^2} \, K_w(\tau_{(a,\ell)^{-1}}(x,y))\varphi(x,y)\dd x\dd y
\end{equation*}
where $K_w\in\caD'(\gR^2)[[\theta]]$. Then, Equation \eqref{eq-retractgeo1} becomes
\begin{equation*}
\int_{\gR^2}\, (D^{(\alpha,\beta)}_X)_{|(a,\ell)}K_w(\tau_{(a,\ell)^{-1}}(x,y))\varphi(x,y)\dd x\dd y=\int_{\gR^2}\, K_w(\tau_{(a,\ell)^{-1}}(x,y)) X^*_{|(x,y)}\varphi(x,y)\dd x\dd y.
\end{equation*}
with
\begin{equation*}
\underline E^*=-y\partial_x,\qquad e_1^*=\partial_y,\qquad Z^*=0.
\end{equation*}
By integration by part, we have
\begin{equation*}
(D^{(\alpha,\beta)}_X)_{|(a,\ell)}K_w(\tau_{(a,\ell)^{-1}}(x,y))=-X^*_{|(x,y)}K_w(\tau_{(a,\ell)^{-1}}(x,y)).
\end{equation*}
We set $\tilde K_w(a,\ell):=K_w(\tau_{(a,\ell)^{-1}}(0,\ee))=K_w(-\ell,\ee e^{2a})$ and\\
 $T_X(\tilde K_w)(a,\ell):=-X^*_{|(x,y)}K_w(\tau_{(a,\ell)^{-1}}(x,y))_{|(x=0,\,y=\ee)}$, i.e. $T_{\underline E}=-\ee e^{-2a}\partial_\ell$ and $T_{e_1}=\frac{\ee}{2}\partial_a$. Within these notations, we obtain a simple form for the equation
\begin{equation*}
D^{(\alpha,\beta)}_X\tilde K_w(a,n)=T_X(\tilde K_w)(a,n).
\end{equation*}

\subsubsection{Resolution}

With partial Fourier transform $\caF_\ell$ in $\ell$ as in \eqref{eq-partfourier}, an explicit computation using $\Lambda_{\underline E}=\alpha e^{-4a}$, $\Lambda_{e_1}=(\beta-\frac{\kappa^2}{8\alpha} \ell)e^{2a}$ and the expression \eqref{eq-star1} of $\star_\theta^1$ gives
\begin{align}
&\caF_\ell D^{(\alpha,\beta)}_{\underline E}\tilde K_w(a,\xi)=-\frac{4i\alpha}{\kappa}e^{-4a}\xi\sqrt{1+\frac{\theta^2}{\kappa^2}\xi^2}\caF_\ell\tilde K_w(a,\xi)\nonumber\\
&\caF_\ell D^{(\alpha,\beta)}_{e_1}\tilde K_w(a,\xi)=e^{2a}\Big[\frac{2i\beta}{\kappa}\xi+\frac{\kappa}{8\alpha}\sqrt{1+\frac{\theta^2}{\kappa^2}\xi^2}(\partial_a+2+2\xi\partial_\xi)+\frac{\theta^2}{8\alpha\kappa}\frac{\xi^2}{\sqrt{1+\frac{\theta^2}{\kappa^2}\xi^2}}\Big]\caF_\ell\tilde K_w(a,\xi).\label{eq-exprD}
\end{align}
By setting $u(a,\xi)=\caF_\ell\tilde K_w(a,\xi)$, we obtain the following system:
\begin{align*}
&\Big(\frac{4\alpha}{\kappa}e^{-4a}\sqrt{1+\frac{\theta^2}{\kappa^2}\xi^2}-\ee e^{-2a}\Big)\xi\, u(a,\xi)=0\\
&\Big(\frac{2i\beta}{\kappa}\xi+\frac{\kappa}{8\alpha}\sqrt{1+\frac{\theta^2}{\kappa^2}\xi^2}(\partial_a+2+2\xi\partial_\xi)+\frac{\theta^2}{8\alpha\kappa}\frac{\xi^2}{\sqrt{1+\frac{\theta^2}{\kappa^2}\xi^2}}-\ee \frac{e^{-2a}}{2}\partial_a\Big)u(a,\xi)=0.
\end{align*}
The first equation tells us that $u$ will be a linear combination of a distribution of support $\xi=0$ and one of support $1+\frac{\theta^2\xi^2}{\kappa^2}=\frac{\kappa^2}{16\alpha^2}e^{4a}$. Then, we can plug this form into the second equation. For more simplicity and to recover the result of section \ref{subsec-retractqu}, we want to find a solution for the constant $\alpha=\frac{\kappa^{\frac32}}{8}\ee$. Then, we can check that
\begin{equation*}
u(a,\xi)=e^{-3a}\Big|\frac{2}{\sqrt\kappa}e^{2a}+\sqrt{\frac{4}{\kappa}e^{4a}-1}\Big|^{\frac{i\beta\sqrt\kappa}{\theta}}\delta\Big(\xi\pm \frac{\kappa}{\theta}\sqrt{\frac{4}{\kappa}e^{4a}-1}\Big)
\end{equation*}
is a solution of these equations, with a freedom in the normalization. An easy computation shows that it corresponds actually to the intertwiner
\begin{equation*}
W_\ee(f)(a,\ell)=\frac{\kappa}{2\pi\theta}e^{-2a}\int_{\gR^2} (1+\eta^2)^{-\frac14} |\sqrt{1+\eta^2}+\eta|^{\frac{i\beta\sqrt\kappa}{\theta}} e^{\frac{i\kappa\ee}{\theta}\eta(\ell-e^{-2a}q)}f\Big(q,\frac{\sqrt\kappa\ee}{2}e^{-2a}\sqrt{1+\eta^2}\Big)\dd \eta\dd q,
\end{equation*}
where the normalization has been set to preserve the function 1. It coincides with the expression \eqref{eq-intertw} if the weight is chosen as $\bfm(a)=|\cosh(2a)+\sinh(2a)|^{\frac{i\beta\sqrt\kappa}{\theta}}=e^{\frac{2i\beta\sqrt\kappa}{\theta}a}$, which is unitary. Note that the star-product $\star^1$ is not affected by this weight $\bfm$: we have $\star_{\ee\theta,\bfm}^1=\star_{\ee\theta}^1$.

\subsection{Method via star-exponential}
\label{subsec-interstarexp}

Let us give a third method to construct the intertwining operator $W$ between $\sharp_\theta$ and $\star_\theta^1$. Knowing the expression of the one $T_{01}$ between $\star_\theta^0$ and $\star_\theta^1$, it suffices to determine an intertwiner $T$ between $\sharp_\theta$ and $\star_\theta^0$. To this aim, we will compute the expression of the star-exponential $E_{\star_\theta^0}(j_1)$ and $E_{\star_\theta^0}(j_2)$ of the coordinates $x=j_1(a,\ell)$ and $y=j_2(a,\ell)$ via the change of charts \eqref{eq-changej} (we omit the $\kappa$ in the notation $j_\kappa$). Then, the usual exponential in the $\Psi_\kappa$ coordinates coincides with the star-exponential for $\sharp_\theta$. Pushed by the intertwiner, it gives the star-exponential of $x=j_1(a,\ell)$ and $y=j_2(a,\ell)$ for the star-product $\star_\theta^0$. Using Fourier transformation with respect to these star-exponential, we can express the intertwiner $T$.

\subsubsection{Star-exponential of the coordinates}

First, let us compute by the method of section \ref{subsec-starexppsi} the expression of $E_{\star_\theta^0}(p(\ell+\beta')e^{2a}+q\frac{\sqrt\kappa}{2}e^{-2a})$, with $\beta',p,q\in\gR$, which is more general as just $E_{\star_\theta^0}(j_1)$ or $E_{\star_\theta^0}(j_2)$. Note that the coordinates are affine in the variable $\ell$ so that we obtain PDE of order 1 for the star-exponential. Let $u_t(a,\ell)$ be a solution of
\begin{equation}
\partial_t u_t=\frac{i}{\theta} (p(\ell+\beta')e^{2a}+\frac{\sqrt\kappa}{2}qe^{-2a})\star_\theta^0 u_t\label{eq-defstarexpcoord}
\end{equation}
with initial condition $u_0(a,\ell)=1$. We apply the functional transform \eqref{eq-transtilde2}, so we get
\begin{equation*}
\partial_t \thatp u_t(\xz,\yz)=-\frac{p}{\kappa}e^{2\xz}(\partial_\xz+1-\frac{i\beta'\kappa}{\theta})\thatp u_t(\xz,\yz) +\frac{i q\sqrt\kappa}{2\theta}e^{-2\xz}\,\thatp u_t(\xz,\yz).
\end{equation*}
Then, with $Z=e^{2\xz}$, we arrive at
\begin{equation*}
\partial_t \thatp u_t=-\frac{pZ}{\kappa}(2Z\partial_Z+1-\frac{i\beta'\kappa}{\theta})\thatp u_t+\frac{i q\sqrt\kappa}{2\theta}\frac{1}{Z}\thatp u_t.
\end{equation*}

As in section \ref{subsec-starexppsi}, we obtain the integral curves $X(t,Z)$ of the vector field $-\frac{2p}{\kappa}Z^2\partial_Z$ with condition $X(0,Z)=Z$:
\begin{equation*}
X(t,Z)=\frac{Z}{1+\frac{2pt}{\kappa}Z}.
\end{equation*}
And the solution of the part of degree 1 of the equation can be written as $\thatp u_t(Z,\yz)=\thatp f(\frac{Z}{1+\frac{2pt}{\kappa}Z},\yz)\rho_t(Z)$ with $\thatp f(Z,\yz)=\frac{2\pi\theta}{\kappa}\delta(\frac12\log(Z)-\yz)$ due to the initial condition. The function $\rho_t$ then satisfies the same equation as $\thatp u$. Performing the change of variables $t':=t$ and $Z':=X(t,Z)=\frac{Z}{1+\frac{2pt}{\kappa}Z}$, we obtain
\begin{equation*}
\partial_{t'} \rho_{t'}(Z')=-\frac{pZ'}{\kappa(1-\frac{2pt'}{\kappa}Z')} (1-\frac{i\beta'\kappa}{\theta})\rho_{t'}(Z') +\frac{i q\sqrt\kappa}{2\theta}\big(\frac{1}{Z'}-\frac{2pt'}{\kappa}\big)\rho_{t'}(Z').
\end{equation*}
With initial condition $\rho_0(Z)\equiv 1$, we find the solution
\begin{equation*}
\rho_t(Z)=\exp\Big(-\frac12(1-\frac{i\beta'\kappa}{\theta})\log(1+\frac{2pt}{\kappa}Z) +\frac{i q\sqrt\kappa}{2\theta}\big(\frac{t}{Z}+\frac{pt^2}{\kappa}\big)\Big).
\end{equation*}
Plugging this expression into $\thatp u_t$ and simplifying, we get
\begin{multline}
E_{\star_\theta^0}(tp(\ell+\beta')e^{2a}+t \frac{\sqrt\kappa}{2}qe^{-2a})=u_t(a,\ell)=\frac{1+\frac{2pt}{\kappa}e^{2a}(\frac{pt}{\kappa}e^{2a}+\sqrt{1+\frac{p^2t^2}{\kappa^2}e^{4a}})}{1+\frac{pt}{\kappa}e^{2a}(\frac{pt}{\kappa}e^{2a}+\sqrt{1+\frac{p^2t^2}{\kappa^2}e^{4a}})}\\
\exp\Big(\frac{i\kappa}{\theta}\ell\log(\frac{pt}{\kappa}e^{2a}+\sqrt{1+\frac{p^2t^2}{\kappa^2}e^{4a}})\Big)
\, \Big(\frac{pt}{\kappa}e^{2a}+\sqrt{1+\frac{p^2t^2}{\kappa^2}e^{4a}}\Big)^{-1+\frac{i\beta'\kappa}{\theta}}\\
\exp\Big(\frac{i q\sqrt\kappa}{2\theta}\big(\frac{te^{-2a}}{\frac{pt}{\kappa}e^{2a}+\sqrt{1+\frac{p^2t^2}{\kappa^2}e^{4a}}}+\frac{pt^2}{\kappa}\big)\Big).\label{eq-starexpcoord}
\end{multline}

\subsubsection{Intertwining operator}

Let us now construct the intertwining operator $T$ between $\sharp_\theta$ and $\star_{\theta}^0$. We have indeed to distinguish the two cases $=\pm 1$ as we learned from previous sections. With a Fourier and a Fourier inverse transformation, for any function $f$ on $\gR^2$, we have
\begin{equation*}
f(x,y)=\frac{1}{(2\pi\theta)^2}\int\, e^{-\frac{i}{\theta}(\xi p+\eta q)} f(\xi,\eta) E_{\sharp_\theta}(px+qy) \dd\xi\dd\eta\dd p\dd q
\end{equation*}
since $E_{\sharp_\theta}(px+qy)=e^{\frac{i}{\theta}(px+qy)}$ for the Moyal product $\sharp_\theta$. So if such an intertwiner $T$ exists, it has to be of the form
\begin{equation}
T(f)(a,\ell):=\frac{1}{(2\pi\theta)^2}\int\, e^{-\frac{i}{\theta}(\xi p+\eta q)} f(\xi,\eta) E_{\star_{\theta}^0}(pT(x)+qT(y))(a,\ell) \dd\xi\dd\eta\dd p\dd q.\label{eq-defT}
\end{equation}
Let us define $T$ as above. We have to choose the values $T(x)$ and $T(y)$ in such a way that $[T(x),T(y)]_{\star_{\theta}^0}=T([x,y]_{\sharp_\theta})$. This condition looks like covariance of the star-product and the following BCH-like formula will be derived from this condition:
\begin{equation}
E_{\star_{\theta}^0}(T(x))\star_{\theta}^0 E_{\star_{\theta}^0}(T(y))=E_{\star_{\theta}^0}(\text{BCH}(T(x),T(y))).\label{eq-bchformT}
\end{equation}
This formula will be the crucial argument to prove that $T$ is an intertwining operator between $\sharp_\theta$ and $\star_{\theta}^0$.

\begin{theorem}
The operator $T$ defined in \eqref{eq-defT} satisfies the equivalence: $T_{01}\circ T$ is $\gS$-equivariant if and only if $T(x)(a,\ell)=(\ell+\beta')e^{2a}$ and $T(y)(a,\ell)=\frac{\sqrt{\kappa}\gamma}{2}e^{-2a}$, for $\beta',\gamma\in\gR$. 

For these values, $T$ is an intertwining operator between $\sharp_\theta$ and $\star_{\theta}^0$.
\end{theorem}
\begin{proof}
First, we notice that
\begin{multline*}
T(x)(a,\ell)=\frac{1}{2\pi\theta}\int\,  e^{-\frac{i}{\theta}\xi p} \xi\,E_{\star_{\theta}^0}(p T(x))(a,\ell) \dd\xi\dd p= \frac{i}{2\pi}\int\, (\partial_p e^{-\frac{i}{\theta}\xi p}) E_{\star_{\theta}^0}(p T(x))(a,\ell) \dd\xi\dd p\\
= -i\theta (\partial_p E_{\star_{\theta}^0}(p T(x)))_{p=0}(a,\ell)=T(x)(a,\ell),
\end{multline*}
which does not impose any constraint on $T(x)$. We have the same argument for $T(y)$.

We recall the action $\tau$  \eqref{eq-actsq} of $\gS$ on $\gR^2$: $\tau_{(a,\ell)}(x,y)=((x+\ell)e^{2a},ye^{-2a})$. With the help of \eqref{eq-inter01}, we have:
\begin{multline*}
(T_{01}\circ T)(\tau_{(a,\ell)}^*f)(a_0,\ell_0)=\frac{1}{(2\pi\theta)^2}\int\, f((\xi+\ell)e^{2a},\eta e^{-2a}) e^{-\frac{i}{\theta}(\xi p+\eta q)}\\
E_{\star_\theta^1}((T_{01}\circ T)(px+qy))(a_0,\ell_0) \dd\xi\dd\eta\dd p\dd q
= \frac{1}{(2\pi\theta)^2}\int\, f(\xi',\eta') e^{-\frac{i}{\theta}(\xi' p'+\eta' q')}\\ E_{\star_\theta^1}\Big(p' e^{2a} (T_{01}\circ T (x)+\ell)+q'e^{-2a}T_{01}\circ T (y)\Big)(a_0,\ell_0) \dd\xi'\dd\eta'\dd p'\dd q'
\end{multline*}
with a change of variable and using $E_{\star_\theta^1}(p' e^{2a}\ell)(a_0,\ell_0)=e^{\frac{i}{\theta}p'e^{2a}\ell}$ (constant with respect to $a_0,\ell_0$). To obtain $(T_{01}\circ T)(f)((a,\ell)\fois(a_0,\ell_0))$ as a result, we have to identify $p'e^{2a} ((T_{01}\circ T) (x)(a_0,\ell_0)+\ell)+q'e^{-2a}(T_{01}\circ T) (y)(a_0,\ell_0)$ with $(T_{01}\circ T)(p'x+q'y)((a,\ell)\fois(a_0,\ell_0))$, for any $p',q'$, since $\star_\theta^1$ and so $E_{\star_\theta^1}$ are $\gS$-invariant. By taking $a_0=\ell_0=0$, we find $(T_{01}\circ T) (x)(a,\ell)=(\ell+\beta')e^{2a}$ and $(T_{01}\circ T)(y)(a,\ell)=\frac{\sqrt{\kappa}\gamma}{2}e^{-2a}$ with free parameters $\beta',\gamma\in\gR$. Due to explicit computations with $T_{01}^{-1}$, we obtain the desired values for $T(x)$ and $T(y)$. In the following, since it is just a renormalization of $q$, we will take $\gamma= 1$.

Note that by an explicit computation, we have
\begin{equation*}
T([x,y]_{\sharp_\theta})=-\frac{i\theta}{\sqrt\kappa}=\Big[(\ell+\beta')e^{2a},\frac{\sqrt\kappa}{2}e^{-2a}\Big]_{\star_\theta^0}=[T(x),T(y)]_{\star_\theta^0}
\end{equation*}
which will permit to show a particular case of the identity \eqref{eq-bchformT}. Indeed, let us compute
\begin{equation*}
E_{\star_\theta^0}\Big(p(\ell+\beta')e^{2a}+\frac{\sqrt\kappa}{2}qe^{-2a}\Big)\star_\theta^0 E_{\star_\theta^0}\Big(p'(\ell+\beta')e^{2a}+\frac{\sqrt\kappa}{2}q'e^{-2a}\Big).
\end{equation*}
It satisfies the equation \eqref{eq-defstarexpcoord} with a change of initial condition: $u_0(a,\ell)=E_{\star_\theta^0}(p'(\ell+\beta')e^{2a}+\frac{\sqrt\kappa}{2}q'e^{-2a})$. By using the same method as in the computation of \eqref{eq-starexpcoord} with another adapted initial condition $\thatp f$, we obtain
\begin{multline}
E_{\star_\theta^0}\Big(p(\ell+\beta')e^{2a}+\frac{\sqrt\kappa}{2}qe^{-2a}\Big)\star_\theta^0 E_{\star_\theta^0}\Big(p'(\ell+\beta')e^{2a}+\frac{\sqrt\kappa}{2}q'e^{-2a}\Big)\\
= E_{\star_\theta^0}\Big((p+p')(\ell+\beta')e^{2a}+\frac{\sqrt\kappa}{2}(q+q')e^{-2a}\Big) e^{\frac{i}{2\theta\sqrt\kappa}(pq'-p'q)},\label{eq-bchformT2}
\end{multline}
which is the identity \eqref{eq-bchformT} applied to the present situation.

Let $f_1,f_2$ be two functions on $\gR^2$. We will check that $T(f_1\sharp_\theta f_2)=T(f_1)\star_\theta^0 T(f_2)$. By using the definition \eqref{eq-defT} of $T$ and the fact that the Fourier transformation changes the Moyal product $\sharp_\theta$ into the convolution
\begin{equation*}
\int\, e^{-\frac{i}{\theta}(\xi p+\eta q)} (f_1\sharp_\theta f_2)(\xi,\eta) \dd\xi\dd\eta=\int f_1(\xi,\eta) f_2\big(\xi-\frac{q}{2\sqrt\kappa},\eta+\frac{p}{2\sqrt\kappa}\big) e^{-\frac{i}{\theta}(\xi p+\eta q)}\dd\xi\dd\eta,
\end{equation*}
we find
\begin{multline*}
T(f_1\sharp_\theta f_2)(a,\ell)=\frac{1}{(2\pi\theta)^2}\int\, e^{-\frac{i}{\theta}(\xi p+\eta q)} f_1(\xi,\eta)
f_2\big(\xi-\frac{q}{2\sqrt\kappa},\eta+\frac{p}{2\sqrt\kappa}\big)\\
 E_{\star_\theta^0}(pT(x)+qT(y))(a,\ell) \dd\xi\dd\eta\dd p\dd q.
\end{multline*}
In the other way, by using \eqref{eq-bchformT2}, we have
\begin{multline*}
T(f_1)\star_\theta^0 T(f_2)(a,\ell)=\frac{1}{(2\pi\theta)^4}\int\, e^{-\frac{i}{\theta}(\xi p+\eta q+\xi' p'+\eta' q')} f_1(\xi,\eta) f_2(\xi',\eta')\\
 E_{\star_\theta^0}((p+p')(\ell+\beta')e^{2a}+\frac{\sqrt\kappa}{2}(q+q')e^{-2a}) e^{\frac{i}{2\theta\sqrt\kappa}(pq'-p'q)}\dd\xi\dd\eta\dd p\dd q\dd\xi'\dd\eta'\dd p'\dd q'.
\end{multline*}
By performing some integrations, we obtain exactly what we wanted to prove:
\begin{multline*}
T(f_1)\star_\theta^0 T(f_2)(a,\ell)=\frac{1}{(2\pi\theta)^2}\int\, e^{-\frac{i}{\theta}(\xi p+\eta q)} f_1(\xi,\eta)
f_2\big(\xi-\frac{q}{2\sqrt\kappa},\eta+\frac{p}{2\sqrt\kappa}\big)\\
 E_{\star_\theta^0}(pT(x)+qT(y))(a,\ell) \dd\xi\dd\eta\dd p\dd q.
\end{multline*}
\end{proof}

From Equations \eqref{eq-starexpcoord} and \eqref{eq-inter01}, we compute that
\begin{multline*}
W_+(f)(a,\ell):=(T_{01}\circ T)(f)(a,\ell)\\
=\frac{1}{(2\pi\theta)^2}\int\, e^{-\frac{i}{\theta}(\xi p+\eta q)} f(\xi,\eta)
 T_{01}(E_{\star_\theta^0}(p(n
 \ell+\beta')e^{2a}+\frac{\sqrt\kappa}{2}qe^{-2a}))(a,n) \dd\xi\dd\eta\dd p\dd q\\
= \frac{\kappa}{2\pi\theta}e^{-2a}\int_{\gR^2} (1+\eta^2)^{-\frac14}e^{\frac{i\kappa}{\theta}\eta(\ell-e^{-2a}q)}f(q,\frac{\sqrt\kappa}{2}e^{-2a}\sqrt{1+\eta^2}) \,  \Big(\eta+\sqrt{1+\eta^2}\Big)^{\frac{i\beta'\kappa}{\theta}}\dd \eta\dd q.
\end{multline*}
To find $W_-$, just replace $\theta$ by $-\theta$ and $q$ by $-q$ in $T_{01}E_{\star_\theta^0}$ in the second line of the above equation. We see that for $\beta'=\ee\frac{\beta}{\sqrt\kappa}$, we obtain exactly the expression \eqref{eq-intertw} once again.

To conclude, we found here by a third method the same family of $\gS$-equivariant intertwining operator $W_\ee$ between $\sharp_\theta$ and $\star_{\ee\theta}^1$, depending on $\beta\in\gR$. This method is completely different method as the two first where shared symmetry $\gS$ was crucial. Here, we used properties of the Moyal product $\sharp_\theta$, BCH-like formula \eqref{eq-bchformT}, as well as explicit computations of the star-exponential of the coordinates $j_1(a,\ell),$ and $j_2(a,\ell)$.

\section{Conclusion}
\label{sec-concl}

\subsection{Global curvature contraction of the Anti-deSitter space}
\label{subsec-contract}

We note from Theorem \ref{thm-wqu} that unitarity of the intertwiner $W$, which means no loss of information, implies that it is valued in two copies of $\gS$:
\begin{equation*}
W:L^2(\gR^2)\to L^2(\gS)\oplus L^2(\gS).
\end{equation*}
In this section, we interpret these two copies as the global curvature contraction of the Anti-deSitter space $AdS_2$. Let us first compute this global contraction. As in the introduction, we consider the real Lie algebra $\kg_t$ given by 
\begin{equation*}
[H,E]=2E,\qquad [H,F]=-2F,\qquad [E,F]=tH,
\end{equation*}
with $t\geq 0$. However, instead of looking at local charts like in the introduction, we consider globally the Anti-deSitter space denoted by $\tilde M_t$, for $t>0$. One way to describe it globally in view of its contraction ($t\to 0$) is to see it as a sphere for the Killing form in the dual Lie algebra $\kg_t^*$ (so as a coadjoint orbit).

Let $(H^*,E^*,F^*)$ be the dual basis of $\kg_t^*$ with respect to $(H,E,F)$. It turns out that the Killing form on $\kg_t$ is given by
\begin{equation*}
\beta=\begin{pmatrix} 8 & 0 & 0\\ 0&0& 4t\\ 0& 4t &0\end{pmatrix}
\end{equation*}
in the basis $(H,E,F)$. So we define the musical isomorphism $X\in\kg_t\mapsto {}^\flat X\in\kg_t^*$ by ${}^\flat X(Y):=\beta(X,Y)$. One has
\begin{equation*}
{}^\flat H=8H^*,\qquad {}^\flat E=4tF^*,\qquad {}^\flat F=4t E^*.
\end{equation*}
Then, the Killing form can be transported to $\kg_t^*$ by the expression $\beta^*({}^\flat X,{}^\flat Y):=\beta(X,Y)$. And the computation gives
\begin{equation*}
\beta^*=\begin{pmatrix} \frac{1}{8} & 0 & 0\\ 0&0& \frac{1}{4t}\\ 0& \frac{1}{4t} &0\end{pmatrix}
\end{equation*}
in the fixed basis $(H^*,E^*,F^*)$.

To obtain the curvature contraction, one also has to increase the radius of the one-sheeted hyperboloid as $\frac{1}{\sqrt{4t}}$. This will correspond to normalize $\beta^*$ in the following way:
\begin{equation*}
\beta_t:=4t\beta^*
\end{equation*}
This scalar product is actually invariant under the coadjoint action and its spheres coincide with the one-sheeted hyperboloids. Namely, if $\tilde M_t$ is defined as the coadjoint orbit $\Ad^*_G(E^*+F^*)$, one has the following simple characterization
\begin{equation*}
\tilde M_t=\{\xi\in\kg_t^*,\quad \beta_t(\xi,\xi)=2\}.
\end{equation*}
To make explicit the link with the notations of the rest of the paper, we give the $\Phi_\kappa$-local chart of (part of) $\tilde M_t$:
\begin{equation*}
\{2\ell H^*+e^{-2a}E^*+e^{2a}(1-\ell^2 t)F^*,\quad a,\ell\in\gR\}.
\end{equation*}

Now, the global curvature contraction is defined as the limit in $t\to0$, so it corresponds to
\begin{equation*}
\tilde M_0:=\{\xi\in\kg_0^*=(\mathfrak{so}(1,1)\ltimes\gR^2)^*,\quad \beta_0(\xi,\xi)=2\},
\end{equation*}
which can be reformulated as $\tilde M_0=\{\alpha H^*+\gamma E^*+\frac{1}{\gamma}F^*,\ \alpha\in\gR,\, \gamma\in\gR^*\}$. $\tilde M_0$ has two connected components, each corresponding to the Poincar\'e coset $M=SO(1,1)\ltimes\gR^2/\gR$ (see Figure \ref{fig-AdS-Poinca}). Let us summarize the above discussion.
\begin{proposition}
\label{prop-contract}
The global curvature contraction of the Anti-deSitter space $AdS_2$ consists in two copies of the Poincar\'e coset (see Figure \ref{fig-AdS-Poinca}).
\end{proposition}

\begin{figure}[!htb]
  \centering
  \includegraphics[height=20cm]{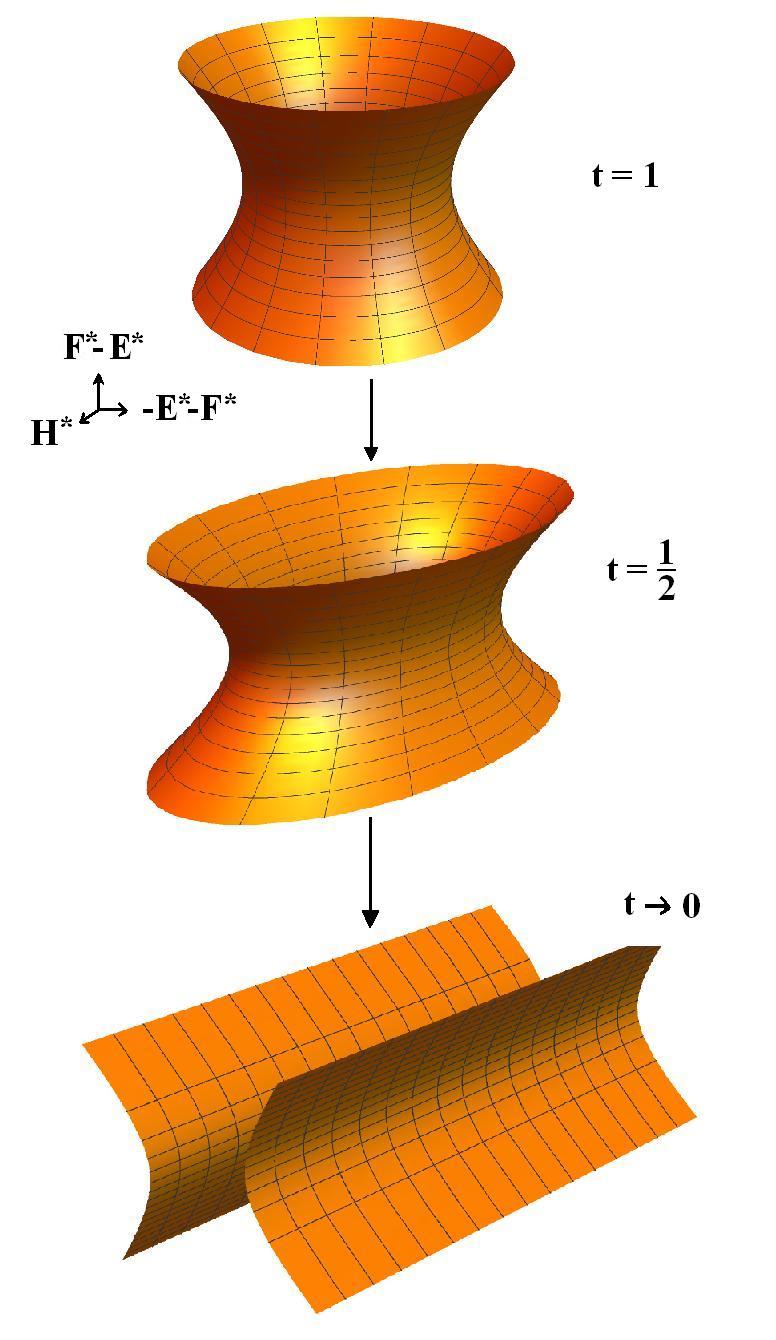}
  \caption[Contour]{\footnotesize{Global curvature contraction ($t\to 0$): Anti-deSitter $\to$ two Poincar\'e cosets.}}
  \label{fig-AdS-Poinca}
\end{figure}


\subsection{Star-exponential on the curvature contraction}

Let us show why we can interpret the two copies of $\gS\simeq M$ (see Remark \ref{rmk-prodpoincare}) in the intertwiner $W$ as the global curvature contraction of $AdS_2$. One can extend in a non-formal way $W$ to polynomials in $x,y$ (see \cite{deGoursac:2014mu}), so that we can evaluate it on the moment maps. Remember that $\lambda_H=\sqrt\kappa (1+2xy)$, $\lambda_E=\sqrt\kappa y$ and $\lambda_F=-\sqrt\kappa x(1+xy)$, so that we have
\begin{align*}
&W_\ee(\lambda_H)=\kappa\ee\ell+\sqrt\kappa(1+\beta),\\
&W_\ee(\lambda_E)=\frac{\kappa}{2}\ee e^{-2a},\\
&W_\ee(\lambda_F)= \frac12e^{2a}\Big((\frac{\theta^2}{2\kappa}-2\beta-\beta^2)\ee-2\sqrt\kappa(1+\beta)\ell-\kappa\ee\ell^2\Big).
\end{align*}
It appears that $\beta=-1$ is a convenient value for this parameter. Then, the coordinates transform as
\begin{equation*}
W_\ee(x)(a,\ell)=e^{2a}(\ell-\frac{\ee}{\sqrt\kappa}),\qquad W_\ee(y)(a,\ell)=\frac{\sqrt\kappa}{2}\ee e^{-2a}.
\end{equation*}

We see that for $\ee=+1$, this corresponds exactly to the change of coordinates defined in Equation \eqref{eq-changej}, so that the first space $\gS$ in the range of $W$ coincides exactly with the one of the coordinate chart $\Phi_\kappa$ (see Equation \eqref{eq-chartphi} and Figure \ref{fig-AdS-habit}).

With $\ee=-1$, it turns out by replacing $x,y$ by their above expression that we get a part of the adjoint orbit described by
\begin{equation*}
-\sqrt\kappa\ell H -e^{2a}(1-\kappa\ell^2)E-\kappa e^{-2a}F
\end{equation*}
in the notations of section \ref{subsec-orbit}. It actually coincides with the chart $\Phi_\kappa'$ (see Figure \ref{fig-AdS-habit}), i.e. with an application of a central symmetry with respect to the origin to the first space $\gS$. In coordinates $x,y$, it corresponds to look at the chart $\Ad_{e^{xE}e^{yF}Z}(H_\kappa)=-\Psi_\kappa(xE+yF)$ with $Z=E-F$. 

In the global contraction process (see section \ref{subsec-contract}) - by identifying adjoint and coadjoint orbits thanks to the Killing form for $t>0$ - one can see that the chart $\Phi_\kappa$, now seen as a part of the coadjoint orbit $\tilde M_t$, goes exactly to the first copy of the Poincar\'e coset 
\begin{equation*}
\tilde M_0^+:=\{\alpha H^*+\gamma E^*+\frac{1}{\gamma}F^*,\ \alpha\in\gR,\, \gamma>0\}\simeq M,
\end{equation*}
while its symmetric counterpart $\Phi_\kappa'$ (also seen in $\tilde M_t$) goes to the second copy
\begin{equation*}
\tilde M_0^-:=\{\alpha H^*+\gamma E^*+\frac{1}{\gamma}F^*,\ \alpha\in\gR,\, \gamma<0\}\simeq M.
\end{equation*}

Therefore, for $\eps=+1$, the space $\gS$ in the range of $W_+$ can be geometrically interpreted as the chart $\Phi_\kappa$ but also as the first connected component $\tilde M_0^+$ of the global contraction $\tilde M_0$. In the same way, for $\eps=-1$, the space $\gS$ in the range of $W_-$ can be geometrically interpreted as the chart $\Phi_\kappa'$ but also as the second connected component $\tilde M_0^-$ of $\tilde M_0$.

Let us now give the interpretation of the star-products. Due to Remark \ref{rmk-prodpoincare}, the natural $SO(1,1)\ltimes \gR^2$-invariant star-product coincides with $\star_\theta^1$ on $M\simeq \gS$. But passing from $\tilde M_0^+\simeq\gS$ to $\tilde M_0^-$ corresponds to a central symmetry, so an overall minus sign, which by Kirillov's orbits method and Weyl type quantization maps \cite{Bieliavsky:2010kg} corresponds to change the sign of the deformation parameter $\theta$ of the star-product. This induces the following definition.
\begin{definition}
The natural $SO(1,1)\ltimes \gR^2$-invariant star-product $\star_\theta$ on the global curvature contraction $\tilde M_0=\tilde M_0^+\cup\tilde M_0^-$ is defined by
\begin{equation*}
f\star_\theta h:= f_+\star_\theta^1 h_+ \,+\, f_-\star^1_{-\theta} h_-,
\end{equation*}
where $f_{\pm}$ is the restriction of $f$ to $\tilde M_0^\pm$ and we use the identification $\tilde M_0^\pm\simeq\gS$.
\end{definition}
Then, it turns out that $(L^2(\tilde M_0),\star_\theta)\simeq (L^2(\gS),\star_\theta^1)\oplus (L^2(\gS),\star_{-\theta}^1)$ as Hilbert algebras, like in the range of $W$.

\medskip

We can now collect the results of section \ref{sec-inter}.
\begin{proposition}
\label{prop-multW}
For any $\beta\in\gR$, the intertwining operator $W=W_{+}\oplus W_{-}: L^2(\gR^2)\to L^2(\gS)\oplus L^2(\gS)$ defined by
\begin{equation*}
W_\ee(f)(a,\ell)=\frac{\kappa}{2\pi\theta}e^{-2a}\int_{\gR^2} (1+\eta^2)^{-\frac14} |\sqrt{1+\eta^2}+\eta|^{\frac{i\beta\sqrt\kappa}{\theta}} e^{\frac{i\kappa\ee}{\theta}\eta(\ell-e^{-2a}q)}f\Big(q,\frac{\sqrt\kappa\ee}{2}e^{-2a}\sqrt{1+\eta^2}\Big)\dd \eta\dd q,
\end{equation*}
is an isomorphism of Hilbert algebras. Therefore, it induces a spatial isomorphism of von Neumann algebras $W:\kM_{\sharp_\theta}(\gR^2)\to\kM_{\star_\theta^1}(\gS)\oplus \kM_{\star_{-\theta}^1}(\gS)\simeq\kM_{\star_\theta}(\tilde M_0)$.
\end{proposition}
\begin{proof}
In Theorem \ref{thm-wqu} was proved that $W$ is a unitary algebra homomorphism for the star-products $\sharp_\theta$ and $(\star_\theta^1,\star_{-\theta}^1)$. Let us show that it is compatible with the complex conjugation:
\begin{multline*}
\overline{W_\ee(f)(a,\ell)}=\frac{\kappa}{2\pi\theta}e^{-2a}\int_{\gR^2} (1+\eta^2)^{-\frac14} |\sqrt{1+\eta^2}+\eta|^{\frac{-i\beta\sqrt\kappa}{\theta}} e^{-\frac{i\kappa\ee}{\theta}\eta(\ell-e^{-2a}q)}\\
\overline{f\Big(q,\frac{\sqrt\kappa\ee}{2}e^{-2a}\sqrt{1+\eta^2}\Big)}\dd \eta\dd q
=W_\ee(\overline f)(a,\ell),
\end{multline*}
by performing the change of variables $\eta\mapsto -\eta$ and using $(\sqrt{1+\eta^2}-\eta)^{-1}=(\sqrt{1+\eta^2}+\eta)$.
\end{proof}

The intertwiner $W$ is defined on the multipliers $\kM_{\sharp_\theta}(\gR^2)$, so that we can push by $W$ the non-formal star-exponential $\caE_{\sharp_\theta}$ of the group $G=SL(2,\gR)$ obtained in section \ref{sec-principal}. We have $W_\ee(E_{\sharp_{\theta}}(t\lambda_X))(a,\ell)=E_{\star_{\ee\theta}^1}(t T_{01}(\lambda_X))(a,\ell)$ and an easy computation gives
\begin{equation*}
T_{01}(\lambda_H)=\lambda_H,\qquad T_{01}(\lambda_E)=\lambda_E,\qquad T_{01}(\lambda_F)=\lambda_F+\frac{\theta^2}{4\kappa}e^{2a},
\end{equation*}
so $W(\caE_{\sharp_\theta})$ is characterized by the expressions \eqref{eq-expr1starexp} and \eqref{eq-intertw}.
\begin{proposition}
The expression
\begin{equation*}
\caE_{\star_\theta}:=W(\caE_{\sharp_\theta}):G\to\kM_{\star_\theta}(\tilde M_0)
\end{equation*}
defines the non-formal star-exponential of the group $G=SL_2(\gR)$ realized on the global curvature contraction $\tilde M_0$ of the Anti-deSitter space $AdS_2$ (see Figure \ref{fig-AdS-Poinca}) with its natural $SO(1,1)\ltimes\gR^2$-invariant star-product $\star_\theta$. It is a continuous group homomorphism for the weak topology of the von Neumann algebra $\kM_{\star_\theta}(\tilde M_0)$.
\end{proposition}

By analytic methods (through the determination of the intertwiner $W$ and to get unitarity), we thus obtained the geometric information that the global curvature contraction coincides with two copies of the Poincar\'e coset. Moreover, we know from \cite{deGoursac:2014mu} that the left or right $\star_\theta$-multiplication by this star-exponential $\caE_{\star_\theta}$ preserves the Fr\'echet algebra $(\caS(\tilde M_0),\star_\theta)$ defined by
\begin{align*}
\caS(\tilde M_0):=\{f\in L^2(\tilde M_0),\quad &\forall X_j,Y_j\in\mathfrak{sl}_2(\gR),\\
& \norm \{\lambda_{X_1},\{\lambda_{X_2},\dots\{\lambda_{X_p},[\lambda_{Y_1},\dots[\lambda_{Y_q},f_+]_{\star_\theta^1}\dots]_{\star_\theta^1}\}_{\star_\theta^1}\dots\}_{\star_\theta^1} \norm<\infty,\\
\text{and }& \norm \{\lambda_{X_1},\{\lambda_{X_2},\dots\{\lambda_{X_p},[\lambda_{Y_1},\dots[\lambda_{Y_q},f_-]_{\star_{-\theta}^1}\dots]_{\star_{-\theta}^1}\}_{\star_{-\theta}^1}\dots \}_{\star_{-\theta}^1} \norm<\infty\}.
\end{align*}

\subsection{Application to Bessel function identities}

Let us finally compare the star-exponential of section \ref{sec-direct} involving Bessel functions with the one of section \ref{sec-principal} involving principal series representation pushed by the intertwiner $W$.

We saw previously that expressions are much simpler in coordinates $(\xz,\yz)$ and $(\xw,\yw)$ instead of $(x,y)$ and $(a,\ell)$. So, we denote by $\caT$ the composition of the transformation $(\xz,\yz)\mapsto (x,y)$ (with partial Fourier transform and change of variables, see section \ref{subsec-starexppsi}) with $T:=T_{01}^{-1}\circ W_+$, which is the intertwiner between $\sharp_\theta$ and $\star_\theta^0$, and with the transformation $(a,\ell)\mapsto (\xw,\yw)$ (see section \ref{subsec-funcdescr}). Then $\caT:\kM_{\that\sharp_\theta}(\gR^2)\to\kM_{\thatp\star_\theta}(\gS)$ intertwines the products $\that\sharp_\theta$ and $\thatp\star_\theta$. Its explicit expression is given by
\begin{equation*}
\caT(\that f)(\xw,\yw)=\frac{\sqrt\kappa}{2\pi\theta} e^{-\xw-\yw} e^{\frac{-i\sqrt\kappa}{\theta}(\xw-\yw)} \int e^{\frac{i\kappa}{2\theta}(e^{-2\xw}\xz-e^{-2\yw}\yz)}\that f(\xz,\yz)\dd\xz\dd\yz.
\end{equation*}
We use now the form of the star-exponential of the group $G=SL(2,\gR)$ in the variables $(\xz,\yz)$ (see \eqref{eq-expr1starexp}):
\begin{equation*}
\caE_{\that\sharp_\theta}(g)(\xz,\yz)=\frac{2\pi\theta}{\sqrt\kappa}|a-c\xz|^{-1+\frac{i\sqrt\kappa}{\theta}}\,\delta\Big(\yz-\frac{d\xz-b}{a-c\xz}\Big),
\end{equation*}
for $g=\begin{pmatrix} a & b \\ c & d \end{pmatrix}$, to obtain
\begin{equation*}
\caE_{\thatp\star_\theta}(g)(\xw,\yw)=\caT(\caE_{\that\sharp_\theta}(g))(\xw,\yw)=e^{-(\xw+\yw)} e^{-\frac{i\sqrt\kappa}{\theta}(\xw-\yw)}\int |a-c\xz|^{-1+\frac{i\sqrt\kappa}{\theta}}\,e^{\frac{i\kappa}{2\theta}(e^{-2\xw}\xz-e^{-2\yw}\frac{d\xz-b}{a-c\xz})}\dd \xz.
\end{equation*}
In particular, with $g=e^{tF}=\begin{pmatrix} 1 & 0 \\ t & 1 \end{pmatrix}$, the above expression satisfies Equation \eqref{eq-fouriersimple}, so by unicity it should coincide with the star-exponential $E_{\thatp \star_\theta}(\frac{i}{\theta}t\lambda_F)$ obtained in Theorem \ref{thm-starexpbessel}. This identification yields a new identity on Bessel functions as we will see.

Before stating the result, let us recall some identities on Bessel functions. Combining \cite[p. 395]{Watson:1966} and \cite[p. 181]{Watson:1966}, we obtain the following property:
\begin{equation*}
\int_{\gR_+} s\,e^{-p^2s^2} J_{\nu}(\alpha s) J_{\nu}(\beta s)\dd s
=-\frac{1}{4i\pi p^2}e^{-\frac{\alpha^2+\beta^2}{4p^2}}\int_{(0+)}^{+\infty} u^{-1-\nu}\,e^{-\frac{\alpha\beta}{4p^2}(u+\frac1u)}\dd u,
\end{equation*}
for $\alpha,\beta\in\gR_+^*$, $\nu, p\in\gC$ such that $|\text{Arg}(p)|<\frac{\pi}{4}$, $-1<\Re(\nu)<1$, and the integration contour is described on Figure \ref{fig-contour}.

\begin{figure}[!htb]
  \centering
  \includegraphics[height=4cm]{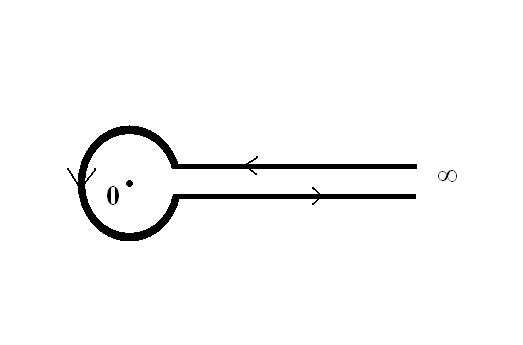}
  \caption[Contour]{\footnotesize{Integration contour $(0+)\to +\infty$.}}
  \label{fig-contour}
\end{figure}


Note that the RHS of the identity involves a well-defined Lebesgue integral on the complex contour, with an exponential decreasing because of the condition $|\text{Arg}(p)|<\frac{\pi}{4}$. We can now show that the comparison between the two star-exponential gives this type of identity but at the singular limit $|\text{Arg}(p)|=\frac{\pi}{4}$, for $\nu\in i\gR^*$, and if we replace the Bessel function $J$ by another function $A$ of Bessel type. Therefore, there will be no exponential decreasing in the second member, only an oscillatory integral.
\begin{theorem}
For $\alpha,\beta\in\gR_+^*$, $\tau\in\gR^*$ and $p\in\gC$ such that $|\text{Arg}(p)|=\frac{\pi}{4}$, we have the following identity:
\begin{equation*}
\int_{\gR_+} \,e^{-p^2s^2} A_{\tau}(\alpha,s) A_{\tau}(\beta,s)\frac{s}{1+4\tau^2 s^2}\dd s
=\frac{1}{4\pi |p^2|}e^{-\frac{\alpha^2+\beta^2}{4p^2}}\int_\gR |\xz|^{-1+i\tau}\,e^{-\frac{\alpha\beta}{4p^2}(\xz+\frac{1}{\xz})}\dd \xz,
\end{equation*}
where we recall that
\begin{equation*}
A_\tau(\alpha,s):=\frac{1}{\sinh(\frac{\pi\tau}{2})} \Re\big( Y_{i\tau}(s\alpha)-2s\tau J_{i\tau}(s\alpha)\big).
\end{equation*}
\end{theorem}
\begin{proof}
Indeed, the identification $E_{\thatp\star_\theta}(\frac{i}{\theta}\lambda_F)$ with $\caT(E_{\that\sharp_\theta}(\frac{i}{\theta}\lambda_F))$ gives
\begin{multline*}
\frac{4\pi}{|\theta|}\int_{\gR_+} \,e^{-\frac{i}{\theta}s^2t} A_{\frac{\sqrt\kappa}{\theta}}( \frac{\sqrt{2\kappa}}{\theta}e^{-\xw},s) A_{\frac{\sqrt\kappa}{\theta}}(\frac{\sqrt{2\kappa}}{\theta} e^{-\yw},s)\frac{s}{1+\frac{4\kappa s^2}{\theta^2}}\dd s\\
=e^{-\frac{i\sqrt\kappa}{\theta}(\xw-\yw)}\int_\gR |1-t\xz|^{-1+\frac{i\sqrt\kappa}{\theta}}\,e^{\frac{i\kappa}{2\theta}(e^{-2\xw}\xz-e^{-2\yw}\frac{\xz}{1-t\xz})}\dd \xz.
\end{multline*}
for $t\in\gR$ and $\xw,\yw\in\gR$. We set $\tau=\frac{\sqrt\kappa}{\theta}$, $\alpha=\sqrt 2\tau e^{-\xw}$ and $\beta=\sqrt 2\tau e^{-\yw}$. After the change of variables $\xz\mapsto \frac{\alpha}{\beta}(t\xz-1)$, we obtain indeed
\begin{equation*}
\frac{4\pi}{|\theta|}\int_{\gR_+} \,e^{-\frac{i}{\theta}s^2t} A_{\tau}(\alpha,s) A_{\tau}(\beta,s)\frac{s}{1+4\tau^2 s^2}\dd s
=\frac{1}{|t|}e^{\frac{i\theta}{4t}(\alpha^2+\beta^2)}\int_\gR |\xz|^{-1+i\tau}\, e^{\frac{i\alpha\beta\theta}{4t}(\xz+\frac{1}{\xz})}\dd \xz,
\end{equation*}
and we set $p^2=\frac{it}{\theta}$.
\end{proof}

\bibliographystyle{hplain}
\bibliography{biblio-these,biblio-perso,biblio-recents}

\end{document}